\documentclass[dvips, 12pt, a4paper]{paper}
\usepackage{amsmath,amssymb,amsthm, dsfont}
\usepackage{color}
\definecolor{violet}{rgb}{0.5,0,0.8}

\theoremstyle{plain}
\newtheorem{thm}{Theorem}[section]

\theoremstyle{plain}
\newtheorem{cor}[thm]{Corollary}

\theoremstyle{plain}
\newtheorem{lemma}[thm]{Lemma}

\theoremstyle{plain}

\theoremstyle{plain}
\newtheorem{proposition}[thm]{Proposition}

\theoremstyle{plain}

\theoremstyle{plain}

\theoremstyle{plain}
\newtheorem{remark}[thm]{Remark}

\begin{document}

\title{Martingale representation\\in progressive enlargement by the reference filtration of a semi-martingale:\\a note on the multidimensional case}
\author{Antonella Calzolari \thanks{Dipartimento di Matematica - Universit\`a di Roma
"Tor Vergata", via della Ricerca Scientifica 1, I 00133 Roma,
Italy }  \and Barbara Torti $^*$} \maketitle
\bigskip
\maketitle

\begin{abstract}
Let $\bold{X}$ and $\bold{Y}$ be an $m$-dimensional $\mathbb{F}$-semi-martingale
and an $n$-dimensional $\mathbb{H}$-semi-martingale respectively
on the same probability space $(\Omega, \mathcal{F}, P)$, both enjoying the strong
predictable representation property.~We propose a martingale
representation result for the square-integrable
$(P,\mathbb{G})$-martingales, where $\mathbb{G}=\mathbb{F}\vee\mathbb{H}$.~As a first
application we identify the biggest possible value of the multiplicity in the sense of Davis and Varaiya of
$\bigvee_{i=1}^d\mathbb{F}^i$, where, 
fixed $i\in(1,\ldots, d)$, $\mathbb{F}^i$ is the reference
filtration of a real martingale $M^i$, which enjoys the $(P,
\mathbb{F}^i)$ predictable representation property.~A second
application falls into the framework of credit risk modeling and
in particular into the study of the progressive enlargement of the
market filtration by a default time.~More precisely, when the
risky asset price is a multidimensional semi-martingale enjoying
the strong
 predictable representation property and the default time satisfies the density hypothesis, we present a new proof of the analogous of the classical Kusuoka's theorem.
\end{abstract}
\textbf{AMS 2010} 60G48, 60G44, 60H05, 60H30, 91G99
\newpage
\section{Introduction}
\noindent On a stochastic basis $(\Omega^\prime,\mathcal{F}^\prime,\mathbb{F}^\prime,P^\prime)$ let $\bold{X^\prime}=(X^{\prime,1},\ldots,X^{\prime,l})$ be an $l$-dimensional $\mathbb{F}^\prime$-semi-martingale,
 which admits at least one equivalent local-martingale measure $Q^\prime$.~Then $\bold{X^\prime}$ enjoys the \textit{$(Q^\prime,\mathbb{F}^\prime)$-(strong)
 predictable representation property ($(Q^\prime,\mathbb{F}^\prime)$-p.r.p.)} when any real $(Q^\prime,\mathbb{F}^\prime)$-local martingale can be written as $m+\boldsymbol{\xi}\bullet \bold{X^\prime}$ where $m$
  is a random variable $\mathcal{F}^\prime_0$-measurable,
 $\boldsymbol{\xi}=(\xi^{1},\ldots,\xi^{l})$ is an $\mathbb{F}^\prime$-predictable process and $\boldsymbol{\xi}\bullet \bold{X^\prime}$   
 is the vector stochastic integral (see \cite{cha-stri94}
 and \cite{shi-che02}).~As well known, this property is equivalent
 to the existence of
a unique, modulo $\mathcal{F}^\prime_0$, equivalent local
martingale measure for $\bold{X^\prime}$ (see Proposition 3.1 in
\cite{ame-be-schw03}).~A particular
 case is when the local martingale measure for $\bold{X^\prime}$ is the unique equivalent martingale measure.~In this case $\mathcal{F}^\prime_0$ is the trivial $\sigma$-algebra (see Theorem 11.2 in \cite{jacod}).
\bigskip\\
In \cite{caltor15}, given a real valued
$\mathbb{F}$-semi-martingale $X$ on a probability space $(\Omega,
\mathcal{F}, P)$ when there exists a unique equivalent martingale
measure $P^X$, we studied the problem of the stability of the
p.r.p.~of $X$ under enlargement of the reference filtration
$\mathbb{F}$.~In particular we assumed the existence of a second
real semi-martingale $Y$ on $(\Omega, \mathcal{F}, P)$ endowed
with a different reference filtration, $\mathbb{H}$, and admitting
itself a unique martingale measure $P^Y$.~We denoted by
$\mathbb{G}$ the filtration obtained by the union of $\mathbb{F}$
and $\mathbb{H}$ and we stated a representation theorem for the
elements of $\mathcal{M}^2(P,\mathbb{G})$, the space of the real
square-integrable $(P,\mathbb{G})$-martingales (see part ii) of
Theorem 4.11 in \cite{caltor15}).~More precisely we assumed the
$(P,\mathbb{G})$-strong orthogonality of the martingale parts $M$
and $N$ of $X$ and $Y$ respectively, and we showed that every
martingale in $\mathcal{M}^2(P,\mathbb{G})$ can be uniquely
represented as sum of an integral with respect to $M$, an integral
with respect to $N$ and an integral with respect to their
quadratic covariation $[M,N]$.~Equivalently we identified $(M, N,
[M,N])$ as a $(P,\mathbb{G})$-\textit{basis of real strongly
orthogonal martingales} (see, e.~g.~\cite{davis2}).~We stress that
$\mathbb{F}$ and $\mathbb{H}$ could be taken larger than the
natural filtration of $X$ and $Y$
respectively.\bigskip\\
In this paper we
deal with the multidimensional version of
 the representation theorem in \cite{caltor15}.~Here $\bold{X}$ is a $(P,\mathbb{F})$ semi-martingale on $(\Omega, \mathcal{F},
 P)$ with values in $\mathbb{R}^m$ and martingale part $\bold{M}$, $\bold{Y}$ is a $(P,\mathbb{H})$ semi-martingale on $(\Omega, \mathcal{F},
 P)$ with values in $\mathbb{R}^n$ and martingale part $\bold{N}$, the initial $\sigma$-algebras $\mathcal{F}_0$ and
$\mathcal{H}_0$ are trivial, $\bold{X}$ enjoys the
$(P,\mathbb{F})$-p.r.p.~and $\bold{Y}$ enjoys the
$(P,\mathbb{H})$-p.r.p.~Finally, for all $i=1,\ldots,m,\;
j=1,\ldots,n$,  we assume the $(P,\mathbb{G})$-strong
orthogonality of the $i$-component of $\bold{M}$, $M^i$, and of
the $j$-component of $\bold{N}$, $N^j$.\bigskip\\Our main result
is that $\mathcal{M}^2(P,\mathbb{G})$
coincides with \textit{the direct sum of three  
stable spaces of square-integrable
martingales}: the stable space generated by $\bold{M}$, the stable space
generated by $\bold{N}$ and the stable space generated by the family of
processes $([M^i,N^j],\; i=1,\ldots,m,\;
j=1,\ldots,n)$ (see \cite{jacod} for the theory of stable
spaces generated by families of martingales and in particular page 114 for their definition and
 Theorem 4.60 at page 143 for their identification as space of vector integrals).~More precisely
any martingale in $\mathcal{M}^2(P,\mathbb{G})$ can be uniquely represented as sum of elements of those stable spaces and
  every pair of elements of any two of those spaces is a pair of real $(P,\mathbb{G})$-strongly orthogonal martingales.~In analogy with the unidimensional case, we can express
the result saying that the triplet of vector processes given by
$\bold{M}$, $\bold{N}$ and any vector martingale obtained by
ordering the family $([M^i,N^j],\; i=1,\ldots,m, j=1,\ldots, n)$
is a \textit{$(P,\mathbb{G})$-basis of multidimensional
martingales}.~\bigskip\\
\noindent Let us present the basic idea and the tools which we use here.\bigskip\\
In order to get the heuristic of the result, but just for this, it
helps to start reasoning with the particular case when $\bold{X}$
and $\bold{Y}$ coincide with their martingale parts $\bold{M}$ and
$\bold{N}$, respectively, and both $\bold{M}$ and $\bold{N}$  have
strongly orthogonal components.~In fact last assumption implies
that under $P$ all vector stochastic integrals with respect to
$\bold{M}$ are componentwise stochastic integrals, that is
$$\boldsymbol{\xi}\bullet \bold{M}=\sum_{i=1}^m \int_0^\cdot \xi^i_t\, dM^{i}_t$$
(see \cite{cha-stri94}
 and Theorem 1.17 in \cite{shi-che02}).~Obviously the same holds for $\bold{N}$.~So that applying
Ito's Lemma to the the product of a $(P,\mathbb{F})$-martingale with a $(P,\mathbb{H})$-martingale and, taking into account the p.r.p.~of
$\bold{M}$ and $\bold{N}$, we realize that for representing all $(P,\mathbb{G})$-martingales we need the family of processes $([M^i,N^j]\; i=1,\ldots,m,\; j=1,\ldots,n)$ in addition to $\bold{M}$ and $\bold{N}$ .\bigskip\\
The proof of our main result when $\bold{X}\equiv \bold{M}$ and
$\bold{Y}\equiv \bold{N}$, not necessarily with pairwise strongly
orthogonal components, is based on two statements of the theory of
stable spaces generated by multidimensional square integrable
martingales, or equivalently of the theory of vector stochastic
integrals with respect to square-integrable martingales.~The two
statements can be roughly resumed as follows.~Given two square
integrable martingales with mutually pairwise strongly orthogonal
components, the stable space generated by one of them is contained
in the subspace orthogonal to the other.~If a multidimensional
 martingale can be decomposed in packets of components mutually pairwise strongly orthogonal, then the stable space generated by this martingale is the direct sum of the stable spaces
 generated by the "packets martingales".~(see Lemma \ref{lemma_orth stable spaces} and Remark (\ref{rem_orth stable spaces})).\bigskip\\
In our framework the assumption that, for all $i=1,\ldots,m$ and $j=1,\ldots,n$, $M^i$ and $N^j$ are
  $(P,\mathbb{G})$-strongly orthogonal martingales
together with the $(P,\mathbb{F})$-p.r.p.~of $\bold{M}$ and the
$(P,\mathbb{H})$-p.r.p.~of $\bold{N}$, allows to show the
$P$-independence of $\mathbb{F}$ and $\mathbb{H}$.~Using this fact
we can prove that, for all $i=1,\ldots,m$ and $j=1,\ldots,n$,
$[M^i,N^j]$ is $(P,\mathbb{G})$-strongly orthogonal to $M^h$,
$h=1,\ldots,m$, and to $N^k$, $k=1,\ldots,n$.~Then the first
general statement above implies that the stable space generated by
$([M^i,N^j]\; i=1,\ldots,m,\; j=1,\ldots,n)$ is contained in the
orthogonal of the subspace generated by $(\bold{M},\bold{N})$,
that is any element of the stable subspace generated by
$([M^i,N^j]\; i=1,\ldots,m,\; j=1,\ldots,n)$ is a real martingale
$(P,\mathbb{G})$-strongly orthogonal to any element of the stable
space generated by $(\bold{M},\bold{N})$.~At the same time using
the $(P,\mathbb{G})$-strong orthogonality of
 $M^i$ and $N^j$, for all $i=1,\ldots,m$ and $j=1,\ldots,n$, we are able to show that $P$ is the unique equivalent martingale measure On $(\Omega, \mathcal{G}_T)$
 for the vector processes $\bold{M}$,
$\bold{N}$ and the real valued processes
$[M^i,N^j]$,\;$i=1,\ldots,m,\;j=1,\ldots,n$.~Then the result
follows by the second general statement above.
  \bigskip\\
   \noindent The well-known formula
   \begin{equation}\label{quadratic_covariation}
[M^i,N^j]_t=\langle M^{c,i},N^{c,j}\rangle_t+\sum_{s\le t}\Delta M^i_s\Delta
N^j_s,
\end{equation}
where $M^{c,i}$ and $N^{c,j}$ are the continuous parts of $M^i$ and $N^j$ respectively, has two immediate consequences.~When $\bold{M}$ and $\bold{N}$ have continuous trajectories, thanks
 to the $P$-independence of $\mathbb{F}$ and $\mathbb{H}$,
 $\bold{M}$ and $\bold{N}$ are enough to represent every $(P,\mathbb{G})$-martingale and the same happens when $\bold{M}$ and $\bold{N}$ admit only totally
 unaccessible jump times.~Therefore, in particular if $\bold{M}$ and $\bold{N}$ are quasi-left continuous martingales then
   $\bold{M}$ and $\bold{N}$ are a $(P,\mathbb{G})$-basis of multidimensional martingales for $\mathcal{M}^2(P,\mathbb{G})$.~Instead, when $\bold{M}$ and $\bold{N}$
    jump simultaneously at accessible jump times, the stable space generated by the covariation
terms has to be added in order to get the representation of $\mathcal{M}^2(P,\mathbb{G})$.\bigskip\\
When $\bold{X}$ and $\bold{Y}$ are not trivial semi-martingales,
that is when they do not coincide with $\bold{M}$ and $\bold{N}$
respectively, under suitable assumptions, we are able to prove
that the representation is the same.~We stress that the approach
is slightly different from that used to handle the unidimensional
case (see Section 4.2 in \cite{caltor15}) and also that the
hypotheses are simpler than those required in that paper.~Here the
question reduces to ask for conditions under which the
$(P^X,\mathbb{F})$-p.r.p.~of $\bold{X}$ and the
$(P^Y,\mathbb{H})$-p.r.p.~of $\bold{Y}$ are equivalent to the
$(P,\mathbb{F})$-p.r.p.~of $\bold{M}$ and the
$(P,\mathbb{H})$-p.r.p.~of $\bold{N}$ respectively, that is to
look for conditions under which the \textit{invariance of the
p.r.p.~under equivalent changes of probability measure} holds
(see, e.g.~Lemma 2.5 in \cite{jean-song15}).~Clearly the involved
changes of measure are two: one of them makes $\bold{X}$ an
$\mathbb{F}$-martingale, the other one makes $\bold{Y}$ an
$\mathbb{H}$-martingale.~Our key assumption is the local
square-integrability of the corresponding Girsanov's
derivatives.~This assumption
 provides in particular the \textit{structure condition} for $\bold{X}$ and $\bold{Y}$.~It is also to note that in our previous paper
  we assumed a bound on the jumps size of $\bold{M}$ and $\bold{N}$, which now follows as a consequence (see \textbf{H4)} in \cite{caltor15}).\bigskip\\
\noindent Moreover here, like in \cite{caltor15}, we obtain also a
second representation result: $\bold{X}$, $\bold{Y}$ and any
ordering of the family $([X^i,Y^j],\; i=1,\ldots,m, j=1,\ldots,
n)$ form a basis of multidimensional martingales for the
$\mathbb{G}$-square-integrable martingales under a new probability
measure on $\mathcal{G}_T$, $Q$, equivalent to
$P|_{\mathcal{G}_T}$.~Indeed, since $\bold{M}$ and $\bold{N}$
enjoy $(P,\mathbb{F})$-p.r.p.~and
$(P,\mathbb{H})$-p.r.p.~respectively, the assumed
$(P,\mathbb{G})$-strong orthogonality of $M^i$ and $N^j$, for all
$i=1,\ldots,m$ and $j=1,\ldots,n$, implies the $P$-independence of
$\mathbb{F}$ and $\mathbb{H}$.~This allows us to construct a
\textit{decoupling probability measure} for  $\mathbb{F}$ and
$\mathbb{H}$ on ${\mathcal{G}_T}$, $Q$, equivalent to $P$ and such
that $\bold{X}$ enjoys the $(Q,\mathbb{F})$-p.r.p.~and $\bold{Y}$
enjoys the $(Q,\mathbb{H})$-p.r.p..~Obviously, the
$Q$-independence of $\mathbb{F}$ and $\mathbb{H}$ implies
$(Q,\mathbb{G})$-strong orthogonality of $X^i-X^i_0$ and
$Y^j-Y^j_0$, for all $i=1,\ldots,m, j=1,\ldots, n)$ and this
allows to obtain the announced representation.
\bigskip\\
\noindent The first application answers to the following
question.~Let a filtration $\mathbb{G}$ be obtained by the union
of a finite family of filtrations,
$(\mathbb{F}^1,\ldots,\mathbb{F}^d)$, such that
$\mathbb{F}^i,\;i=1,\ldots, d,$ is the reference filtration of a
real martingale $M^i$.~Let us assume that $M^i$ enjoys the
$(P,\mathbb{F}^i)$-p.r.p..~Can we determine a
$(P,\mathbb{G})$-basis of real martingales?~Here we give
conditions that make independent the filtrations
$\mathbb{F}^1,\ldots,\mathbb{F}^d$ and as a consequence allow to
 identify a $(P,\mathbb{G})$-basis.~We also underline the link with the notion of multiplicity of a filtration (see \cite{davis1}).\bigskip\\
\noindent The second application falls into the framework of
mathematical finance and more precisely in the reduced form
approach of credit risk modeling.~In \cite{ca-jean-za13}, under
completeness of the reference market and \textit{density
hypothesis} for the default time, the authors got the martingale
representation on the full market under the historical measure
(see also \cite{jean-lecam09} for a similar result).~Our theorem
provides a new proof of that result under slightly different
hypotheses.~Indeed we allow the risky asset price to be a
multidimensional semi-martingale, $\mathbf{S}$ and we assume
 the \textit{immersion property} under the historical measure $P$ of the market
 filtration $\mathbb{F}$ into the filtration $\mathbb{G}$ defined at time $t$ by
 $$\mathcal{G}_t:=\cap_{s> t}\mathcal{F}_s\vee\sigma(\tau\wedge s).$$
One of the key points of our result is, like in
\cite{ca-jean-za13}, the existence of a decoupling measure.~This
measure preserves either the law of $S$  and in particular of its
martingale part $M$, or the law of the \textit{compensated default
process} $H$ defined at time $t$ by
$$H_t:=\mathbb{I}_{\{\tau\leq t\}}-\int_0^{\tau\wedge t}\frac{dF_u}{1-F_u},$$
where $F$ denotes the continuous distribution function of $\tau$.~The last fact
joint with some technical conditions implies that $M$ and $H$  enjoy the
p.r.p.~under the decoupling measure as well under $P$.~From our theorem we immediately derive that
$M$ and $H$ are a basis of multidimensional martingales for the filtration $\mathbb{G}$
under the decoupling measure.~In fact by the density hypothesis the default time does
not coincide with any jump time of
the asset price with positive probability.~Then the invariance property
of the p.r.p.~under equivalent changes of measure allows to establish the representation.\bigskip\\
This
note is organized as follows.~In
\mbox{Section~\ref{sec:sett-not}}, we introduce the notations
and
 some basic definitions, we state the hypotheses, we discuss their consequences and we derive a fundamental ingredient for our result, that is
  the p.r.p.~under $P$ for $\bold{M}$ and $\bold{N}$ with respect to $\mathbb{F}$ and $\mathbb{H}$ respectively.~\mbox{Section~\ref{sec:adding-semimg}}
 is devoted to the main result.~\mbox{Section~\ref{sec:applications}} contains the
 applications.
\section{Setting and hypotheses}\label{sec:sett-not}
Let us fix some notations used in all the
paper.\vspace{0.5em}\\Let $T$ be a finite horizon.~Let
$\bold{S}=(\bold{S}_t)_{t\in[0,T]}=\left((S^1_t,...,S^l_t)\right)_{t\in[0,T]}$
be a c\`adl\`ag square-integrable $l$-dimensional semi-martingale
on a filtered probability space, $(\Omega, \mathcal{A},
\mathbb{A}, R)$, with $\mathbb{A}=(\mathcal{A}_{t\in [0,T]})$
under usual conditions.~We will denote by
$\mathbb{P}(\bold{S},\mathbb{A})$ the set of martingale measures
for $\bold{S}$
on $(\Omega, \mathcal{A}_T)$ equivalent to $R|_{\mathcal{A}_T}$.\bigskip\\
Let $Q\in\mathbb{P}(\bold{S},\mathbb{A})$.~We will denote by
$\mathcal{L}^2(\bold{S}, Q,\mathbb{A})$ the set of the
$\mathbb{A}$-predictable $l$-dimensional processes
$\boldsymbol{\xi}=\left(\boldsymbol{\xi}_t\right)_{t\in[0,T]}=\left((\xi^1_t,...,\xi^l_t)\right)_{t\in[0,T]}$ such that

$$E^Q\left[\int_0^T\boldsymbol{\xi}^{tr}_t\,C^{\boldsymbol{S}}_t\,\boldsymbol{\xi}_t\,dB^{\boldsymbol{S}}_t\right]<+\infty,$$
where
\begin{align}\label{def-B and C}B^{\bold{S}}_t:=\sum_{i=1}^l \langle S^i\rangle_t^{Q,\mathbb{A}},\;\;\;\;\;\;c^{\bold{S}}_{ij}(t):=\frac{d\langle S^i,S^j\rangle_t^{Q,\mathbb{A}}}{dB^{\bold{S}}_t},\;\;i,j\in (1,\ldots,l)\end{align}
with $[S^i,S^j]$  and $\langle S^i,S^j\rangle^{Q,\mathbb{A}}$,
$i,j\in (1,\ldots,l)$ the quadratic covariation process and the
sharp covariation process of
$S^i$ and $S^j$ respectively.\bigskip\\
Following \cite{shi-che02} we will endow $\mathcal{L}^2(\bold{S},
Q,\mathbb{A})$ with the norm
\begin{equation}\label{def-norm}\|\boldsymbol{\xi}\|^2_{\mathcal{L}^2(\bold{S}, Q, \mathbb{A})}:=E^Q\left[\int_0^T\,\boldsymbol{\xi}^{tr}_t\,C^{\bold{S}}_t\,\boldsymbol{\xi}_t\,dB^{\bold{S}}_t\right].\end{equation}
It is possible to prove that
\begin{equation}\label{def-norm2}\|\boldsymbol{\xi}\|^2_{\mathcal{L}^2(\bold{S}, Q, \mathbb{A})}= E^Q\left[[\boldsymbol{\xi}\bullet \bold{S}]_T\right ]\end{equation}
(see formula (3.5) in
\cite{shi-che02}) where $\boldsymbol{\xi}\bullet \bold{S}$ denotes the vector stochastic integral of $\boldsymbol{\xi}\in\mathcal{L}^2(\bold{S}, Q,\mathbb{A})$ with respect to $\bold{S}$.\\
We recall that $\boldsymbol{\xi}\bullet \bold{S}$ is a
one-dimensional process and therefore different from the vector
$(\int_0^\cdot\xi^{1}_t\,dS^{1}_t,\ldots ,
\int_0^\cdot\xi^{l}_t\,dS^{l}_t)$.~Moreover
$\boldsymbol{\xi}\bullet \bold{S}$ coincides with $\sum_{i=1}^l
\int_0^\cdot\xi^{i}_t\,dS^{i}_t$, when $\bold{S}$ has pairwise
$(Q,\mathbb{A})$-strongly orthogonal components (see
\cite{cha-stri94} and \cite{shi-che02}). As noted by Chernyi and
Shiryaev, unless the construction of the vector stochastic
integral is a bit complicated, this notion provides the closeness
of the space
of stochastic integrals.~The notion of componentwise stochastic integral in general does not (for a detailed discussion see \cite{shi-che02}). \bigskip\\
~We will set
$$K^2(\Omega,\mathbb{A},Q,\bold{S}):=\left\{(\boldsymbol{\xi}\bullet \bold{S})_T,\, \boldsymbol{\xi}\in\mathcal{L}^2(\bold{S},Q,\mathbb{A})\right\}.$$
We recall that, when $\mathcal{A}_0$ is trivial,
$\mathbb{P}(\bold{S},\mathbb{A})$ is a singleton, more precisely
$$\mathbb{P}(\bold{S},\mathbb{A})=\{P^\bold{S}\},$$ if and only if $\bold{S}$ enjoys the
$(P^{\bold{S}},\mathbb{A})$-p.r.p.~that is if and only if each $H$ in
$L^2(\Omega,
 \mathcal{A}_T, P^{\bold{S}})$ can be represented $P^{\bold{S}}$-a.s.~, up to an additive constant, as vector stochastic integral with
 respect to $\bold{S}$ that is
\begin{align*}
H=H_0+(\boldsymbol{\xi}^H\bullet \bold{S})_T,
\end{align*}
with $H_0$ a constant  and $\boldsymbol{\xi}^H\in \mathcal{L}^2(\bold{S},
P^{\bold{S}},\mathbb{A})$.~More briefly, $\mathbb{P}(\bold{S},\mathbb{A})=\{P^{\bold{S}}\}$ if and only if
\begin{equation}\label{natural-completeness-2}
L_0^2(\Omega,
 \mathcal{A}_T, P^{\bold{S}})=K^2(\Omega,\mathbb{A},P^{\bold{S}},\bold{S}),\end{equation}
where $L_0^2(\Omega,
 \mathcal{A}_T, P^{\bold{S}})$ is the set of all real centered $P^{\bold{S}}$-square integrable
$\mathcal{A}_T$-measurable random variables
(see \cite{ha-pli2}).~\bigskip\\
We will
indicate by $\mathcal{M}^2(R,\mathbb{A})$ the set of all real square integrable
$(R,\mathbb{A})$-martingales on $[0,T]$, which is a Banach space with the norm
 \begin{equation}\label{def-norm3}\|S\|^2_{\mathcal{M}^2(R,\mathbb{A})}=E^R[S^2_T]\end{equation}
 (see \cite{prott})\footnote{by Jensen inequality   $S^2$ is as a sub-martingale so that $sup_{t\leq T}E[S^2_t]\leq E[S^2_T]< +\infty$ and by Doob's
 inequality $E[sup_{t\leq T} S^2_t]< +\infty$.~An equivalent norm on $\mathcal{M}^2(R,\mathbb{A})$ is  $$E[sup_{t\leq T} S^2_t]$$ (see pages 26, 27 in \cite{jacod})}.\bigskip\\
Finally following \cite{jacod} we will denote by $\mathcal{Z}^2(\boldsymbol{\mu})$
the stable space generated by a finite set of square-integrable martingales
$\boldsymbol{\mu}\subset \mathcal{M}^2(R,\mathbb{A})$.~The general element of $\mathcal{Z}^2(\boldsymbol{\mu})$ is a vector stochastic
integral with respect to $\boldsymbol{\mu}$ (see Theorem 4.60 page 143 in \cite{jacod}).\bigskip\\
Now let us introduce the general setup of our
result.\vspace{0.5em}\\Given a probability space $(\Omega,
\mathcal{F},P)$, a finite time horizon $T\in (0,+\infty)$ and two filtrations
$\mathbb{F}$ and $\mathbb{H}$ under standard conditions and with
$\mathcal{F}_T\subset \mathcal{F}$ and $\mathcal{H}_T\subset
\mathcal{F}$, we consider two square-integrable semi-martingales
and more precisely an $m$-dimensional
 $(P,\mathbb{F})$-semi-martingale $X$ and an $n$-dimensional $(P,\mathbb{H})$-semi-martingale
$Y$ with canonical decomposition
 \begin{equation}\label{semi-martingale}
 \bold{X}=\bold{X}_0+\bold{M}+\bold{A},\;\;\;\;\;\; \bold{Y}=\bold{Y}_0+\bold{N}+\bold{D}
 \end{equation}
such that
 \begin{equation}\label{eq-int-cond-X}
 E^P\left[\|\bold{X}_0\|^2_{\mathbb{R}^m}+\sum^m_{i,j=1}[M^i,M^j]_T+\sum^m_{i=1}|A^i|^2_T\right]<+\infty
 \end{equation}
and
 \begin{equation}\label{eq-int-cond-Y}
 E^P\left[\|\bold{Y}_0\|^2_{\mathbb{R}^n}+\sum^n_{i,j=1}[N^i,N^j]_T+\sum^n_{i=1}|D^i|^2_T\right]<+\infty.
 \end{equation}
Here $\bold{M}=(M^1,...,M^m)$ is an $m$-dimensional
$(P,\mathbb{F})$-martingale with $M^i\in
\mathcal{M}^2(P,\mathbb{F})$ for all $i=1,\ldots,m$,
$\bold{A}=(A^1,...,A^m)$ is an $m$-dimensional
$\mathbb{F}$-predictable process of finite variation,
$\bold{M}_0=\bold{A}_0=0$ and $|A^i|$ denotes the total variation
process of $A^i$.~Note that by the integrability condition
(\ref{eq-int-cond-X}) it follows
\begin{equation}\label{regolarità X}
E^P\Big[\sup_{t\in[0,T]}\|\bold{X}_t\|^2_{\mathbb{R}^m}\Big]<+\infty.
\end{equation}
Analogous considerations hold for $\bold{Y}$.\bigskip\\
We assume that the sets  $\mathbb{P}(\bold{X},\mathbb{F})$ and $\mathbb{P}(\bold{Y},\mathbb{H})$ are singletons and more precisely\vspace{0.5em}\\
\textbf{A1)}\;\textit{$\mathbb{P}(\bold{X},\mathbb{F})=\{P^{\bold{X}}\},$}
\;\textit{$\mathbb{P}(\bold{Y},\mathbb{H})=\{P^{\bold{Y}}\}$}.\vspace{1em}\\
We introduce the Radon-Nikodym derivatives
$$L^{\bold{X}}_t:=\frac{dP^{\bold{X}}}{dP|_{\mathcal{F}_t}},\;\;\;\;L^{\bold{Y}}_t:=\frac{dP^{\bold{Y}}}{dP|_{\mathcal{H}_t}}$$
and their inverses
$$\widetilde{L}_t^{\bold{X}}:=\frac
1{L^{\bold{X}}_t}=\frac{dP|_{\mathcal{F}_t}}{dP^{\bold{X}}},\;\;\;
\;\widetilde{L}^{\bold{Y}}_t:=\frac{1}{L^{\bold{Y}}_t}=\frac{dP|_{\mathcal{H}_t}}{dP^{\bold{Y}}}.$$
Then we require the following regularity conditions on them\vspace{0.5em}\\
\textbf{A2)}\;\begin{equation}\label{ip:regularity
densities}L^{\bold{X}}_T\in
L^2_{loc}(\Omega,\mathcal{F}_T,P),\;\;\;L^{\bold{Y}}_T\in
L^2_{loc}(\Omega,\mathcal{H}_T,P).\end{equation}
\vspace{1em}\\
Let us now discuss the consequences of our assumptions.\bigskip\\
Hypothesis \textbf{A1)} implies that $\mathcal{F}_0$ and $\mathcal{H}_0$ are trivial and that $\bold{X}$ enjoys the $(P^{\bold{X}}, \mathbb{F})$-p.r.p.~and
 $\bold{Y}$ enjoys the $(P^{\bold{Y}}, \mathbb{H})$-p.r.p.~(see \cite{ha-pli2}).\bigskip\\
Using Hypothesis \textbf{A2)} we derive the structure condition for $\bold{X}$ and $\bold{Y}$  and in particular the existence of an
$\mathbb{F}$-predictable $m$-dimensional process
$\boldsymbol{\alpha}=(\alpha^1_t,\ldots,\alpha^m_t)_{t\in [0,T]}$ and an $\mathbb{H}$-predictable process
$\boldsymbol{\delta}=(\delta^1_t,\ldots,\delta^n_t)_{t\in [0,T]}$
 such that for all $i\in (1,\ldots, m)$ and $j\in (1,\ldots, n)$
\begin{equation}\label{struct-cond}A^i_t=\int_0^t\alpha_s^i\,d\langle M^i\rangle^{P,\mathbb{F}}_s,\;\;\;\;\;D^j_t=\int_0^t\delta^j_s\,d\langle N^j \rangle^{P,\mathbb{H}}_s
\end{equation}
(see e.g. Definition 1.1 and Theorem 2.2. in \cite{cho-stri96}).\bigskip\\Moreover it holds
$$\alpha^i\in L^2_{loc}(P\times d\langle M^i\rangle^{P,\mathbb{F}}),\;\;\delta^j\in L^2_{loc}(P\times d\langle N^j\rangle^{P,\mathbb{H}})$$ that is,
following the notations in \cite {jacod}, $\alpha^i\in L^2_{loc}(M^i)$ and $\delta^j\in L^2_{loc}(N^j)$
(see Proposition 4 in \cite{schweizer92} or Theorem 1 in \cite{schw95}).\bigskip\\
Moreover the square-integrability of $\bold{X}$ and $\bold{Y}$ together with assumption \textbf{A2)} imply that for all $i=1,\ldots, m$ and for all $j=1,\ldots, n$, for all $t\in [0,T]$
\begin{equation}\label{ip:regularity inverse densities} \langle \widetilde{L}^{\bold{X}}_t, X^i_t\rangle^{P^{\bold{X}},\mathbb{F}} ,\;\;\;
 \langle \widetilde{L}^{\bold{Y}}_t, Y^j_t
\rangle^{P^{\bold{Y}},\mathbb{H}}\end{equation} exist (see
e.g.~VII 39 in \cite{del-me-b}). \footnote{from (\ref{regolarità
X}) it follows that
$\sup_{t\in[0,T]}\|\bold{X}_t\|_{\mathbb{R}^m}<+\infty$,
$P$-a.s.~and $P^X$-a.s..Therefore, for any divergent sequence
$\{c_n\}_{n\in \mathbb{N}}$ of real number, if
$$T_n:=\inf\left\{{t\in[0,T]}: \sup_{s\leq t}\|\bold{X}_s\|_{\mathbb{R}^m}> c_n\right\},$$ then
$\{T_n\}_{n\in \mathbb{N}}$ is a divergent sequence of stopping
times such that $\sup_{t\in[0,T]}\|\bold{X}_{T_n\wedge
t}\|_{\mathbb{R}^m}< c_n$, that is $\bold{X}$ is locally bounded}
\vspace{0.5em}\\
Finally \textbf{A1)} and \textbf{A2)} allow to transfer the
p.r.p.~from $\bold{X}$ to its martingale part,  $\bold{M}$, and
from $\bold{Y}$ to its martingale part,  $\bold{N}$.~Indeed we can
announce the following result.
\begin{proposition}\label{prop-pred-mart-rep}
Let \textbf{A1)} and \textbf{A2)} be verified.~Then $\bold{M}$ enjoys
the $(P,\mathbb{F})$-p.r.p..
\end{proposition}
\begin{proof}
Set
$$
\tilde{X^i_t}=X^i_t-\int_0^t\frac{1}{\widetilde{L}^X_{s^-}}\;d\langle
\widetilde{L}^{\bold{X}}, X^i
\rangle^{P^{\bold{X}},\mathbb{F}}_{s},\;\;\;i=1,\ldots,m.
$$
Then by Lemma 2.4 in \cite{jean-song15} the process
$\bold{\tilde{X}}=(\tilde{X^1_t},\ldots,\tilde{X^m_t})_{t\in[0,T]}$ is a
$(P, \mathbb{F})$-local martingale which enjoys the
$(P,\mathbb{F})$-p.r.p.~.Moreover it coincides with $\bold{M}$.~In fact,
fixed $i\in (1,\ldots,m)$,
$$
\tilde{X^i_t}-M^i_t=X^i_0+\int_0^t\alpha^i_s\,d\langle
M^i\rangle^{P,\mathbb{F}}_s-\int_0^t\frac{1}{\widetilde{L}^{\bold{X}}_{s^-}}\;d\langle\widetilde{L}^{\bold{X}}, X^i \rangle^{P^{\bold{X},\mathbb{F}}}_s
$$
is a predictable $(\mathbb{F},P)$-local martingale, so it has to
be continuous (see Theorem 43 Chapter IV in \cite
{kall97}).~Moreover it has finite variation, so that it is necessarily
null.
\end{proof}
As in  formula  (\ref{def-B and C}), for any fixed $t\in [0,T]$,
consider the matrix $C^{\bold{M}}_t$ with generic element defined
by
$$c^\mathbf{M}_{ij}(t):=\frac{d\langle M^i,M^j\rangle_t^{P,\mathbb{F}}}{dB^{\bold{M}}_t}$$
with
 $$B^{\bold{M}}_t:=\sum_{i=1}^m \langle
 M^i\rangle_t^{P,\mathbb{F}}.$$
\begin{cor}\label{cor-boundedness-jumps}
Let $\hat{\lambda}$ be defined
by
 \begin{align*}
  C^{\bold{M}}_t\,\boldsymbol{\hat{\lambda}}_t:=\boldsymbol{\gamma}_t,\;\;0\leq t\leq T
 \end{align*}
 where
$$\gamma^i_t:=\alpha^i_t\,c^M_{ii}(t).$$
Then
\begin{itemize}
\item [i)]
 $\boldsymbol{\hat{\lambda}}^{tr}\Delta \bold{M}<1$;
\item [ii)]
 $P^\mathbf{X}$ coincides with the minimal martingale measure for $\bold{X}$.
\end{itemize}
\end{cor}
\begin{proof}
 By hypotheses \textbf{A1)} and \textbf{A2)} $L^{\bold{X}}$ is a locally square integrable strict martingale density under $P$, that is $L^{\bold{X}}\in \mathcal{M}^2_{loc}(P,\mathbb{F})$, so that for all $t\in [0,T]$
 $$L^{\bold{X}}_t=\mathcal{E}\left(- (\boldsymbol{\hat{\lambda}}\,\bullet\,\bold{M})_t + V_t\right)$$
or equivalently $L^{\bold{X}}$
solves the equation
\begin{equation*}
Z=1-Z\boldsymbol{\hat{\lambda}}\,\bullet\,\bold{M} + V,
\end{equation*}
where $V$ is a $(P,\mathbb{F})$-local martingale with real values, null at zero and $(P,\mathbb{F})$-strongly orthogonal to $M^i$ for each $i\in(1,\ldots, m)$  (see Theorem 2.2. in \cite{cho-stri96} or Theorem 1 in \cite{schw95}).\\
Previous proposition forces $V$ to be null so that for all $t\in [0,T]$
$$L^{\bold{X}}_t=\rm{exp}\left(- (\boldsymbol{\hat{\lambda}}\,\bullet\,\bold{M})_t- \sum_{i,j=1}^{m}\,\int_0^t\,\hat{\lambda}^{i}_s
\hat{\lambda}^j_s\,d\langle M^{c,i}, M^{c,j}\rangle_s \right)\prod_{0\leq s\leq t} (1-\boldsymbol{\hat{\lambda}}^{tr}_s\Delta \bold{M}_s)\rm{e}^{\boldsymbol{\hat{\lambda}}^{tr}_s\Delta \bold{M}_s}$$
where $\bold{M}^c=(M^{c,1},\ldots, M^{c,m})$ is the continuous
martingale part of $\bold{M}$.~Since $L^{\bold{X}}$ by assumption
is the derivative of an equivalent change of measure,
then it has to be strictly positive so that part i) follows.\bigskip\\
Proposition 3.1 in \cite{an-str93} proves part ii).
\end{proof}
\begin{remark}
We recall that condition $\boldsymbol{\hat{\lambda}}^{tr}\Delta
\bold{M}<1$ doesn't follow by the existence of a strict martingale
density (see Section 4 in \cite{ans_str92} for a counter-example
in the one-dimensional case, where the condition turns into
$\alpha\Delta M<1$).~Only when $P(\bold{X},\mathbb{F})$ is a
singleton the condition necessarily follows.
\end{remark}
\section{Two bases of martingales}\label{sec:adding-semimg}
In this section we present the multidimensional version of Theorem 4.11 in \cite{caltor15}.~The key assumption is\vspace{0.5em}\\
\textbf{A3)}\;for any $i\in (1,\ldots,m)$ and $j\in (1,\ldots,n)$,
$M^i$ and $N^j$ are real $(P,\mathbb{G})$-strongly orthogonal
martingales, where
$$\mathbb{G}:=\mathbb{F}\vee\mathbb{H},$$
or equivalently for any $i\in (1,\ldots,m)$ and $j\in
(1,\ldots,n)$, the process $[M^i,N^j]$ is
a real uniformly integrable $(P,\mathbb{G})$-martingale  with null initial value.\bigskip\\
Let us denote by $\bold{[M,N]^{V}}$ the process
$$\left([M^1,N^1],\ldots , [M^1,N^n], [M^2,N^1],\ldots , [M^2,N^n],\ldots , [M^m,N^1],\ldots , [M^m,N^n]\right).$$
Then, under assumption \textbf{A3)}, $\bold{[M,N]^{V}}$ is a
$(P,\mathbb{G})$-martingale with values in $\mathbb{R}^{mn}$.\vspace{0.5em}\\
Before announcing the main theorem we state a general result.
\begin{lemma}\label{lemma_orth stable spaces}
On a filtered probability space $(\Omega,
\mathcal{A},\mathbb{A},R)$ let consider two processes
 $\boldsymbol{\mu}=(\mu^1,\ldots, \mu^r)$ and
$\boldsymbol{\mu^\prime}=(\mu^{\prime,1},\ldots, \mu^{\prime,s})$
 such that, for all $i=1,\ldots, r$ and $j=1,\ldots, s$,   $\mu^i$ and $\mu^{\prime,j}$ are
$(R,\mathbb{A})$-strongly orthogonal real martingales in
$\mathcal{M}^2(R,\mathbb{A})$.~Then
\begin{itemize}
\item [i)] for any $\boldsymbol{\xi}\in \mathcal{L}^2(\boldsymbol{\mu}, R, \mathbb{A})$ and
$\boldsymbol{\eta}\in \mathcal{L}^2(\boldsymbol{\mu^\prime}, R,
\mathbb{A})$, the processes $\boldsymbol{\xi}\bullet
\boldsymbol{\mu}$ and $\boldsymbol{\eta}\bullet
\boldsymbol{\mu^\prime}$,  are real $(R,\mathbb{A})$-strongly
orthogonal martingales, that is $\boldsymbol{\xi}\bullet
\boldsymbol{\mu} \cdot \boldsymbol{\eta}\bullet
\boldsymbol{\mu^\prime}$ is a real uniformly integrable
$(R,\mathbb{A})$-martingale with null initial value.
\item [ii)] Moreover let $\boldsymbol{\mu^{\prime\prime}}=(\mu^{\prime\prime,1},\ldots, \mu^{\prime\prime,w})$
be a $(R,\mathbb{A})$-martingale such that, for all fixed $h\in
(1,\ldots, w)$, either the real processes $\mu^{\prime\prime,h}$
and $\mu^{i}$, for every $i=1,\ldots, r$, or the real processes
$\mu^{\prime\prime,h}$ and $\mu^{\prime,j}$, for every
$j=1,\ldots, s$, are $(R,\mathbb{A})$-strongly orthogonal
martingales in $\mathcal{M}^2(R,\mathbb{A})$.\\Then for all
$\boldsymbol{\Theta}\in
\mathcal{L}^2((\boldsymbol{\mu},\boldsymbol{\mu^\prime},
\boldsymbol{\mu^{\prime\prime}}), R, \mathbb{A})$ there exists a
unique triplet $\boldsymbol{\xi}, \boldsymbol{\xi^\prime},
\boldsymbol{\xi^{\prime\prime}}$ with $\boldsymbol{\xi}\in
\mathcal{L}^2(\boldsymbol{\mu}, R, \mathbb{A})$,
$\boldsymbol{\xi^\prime}\in \mathcal{L}^2(\boldsymbol{\mu^\prime},
R, \mathbb{A})$ and $\boldsymbol{\xi^{\prime\prime}}\in
\mathcal{L}^2(\boldsymbol{\mu^{\prime\prime}}, R, \mathbb{A})$
such that
$$\boldsymbol{\Theta}\bullet (\boldsymbol{\mu},\boldsymbol{\mu^\prime}, \boldsymbol{\mu^{\prime\prime}}) =
\boldsymbol{\xi}\bullet \boldsymbol{\mu} + \boldsymbol{\xi^{\prime}}\bullet \boldsymbol{\mu^\prime} + \boldsymbol{\xi^{\prime\prime}}\bullet \boldsymbol{\mu^{\prime\prime}}.$$
\end{itemize}
\end{lemma}
\begin{proof}
As far as the first statement is concerned, fixed $i\in (1,\ldots,
r)$, $\mu^i$ is $(R,\mathbb{A})$-strongly orthogonal to all
elements of $\mathcal{Z}^2(\boldsymbol{\mu^\prime})$ (see point a)
of Theorem 4.7, page 116 in \cite{jacod}), that is to every vector
integral with respect to $\boldsymbol{\mu^\prime}$ of a process in
$\mathcal{L}^2(\boldsymbol{\mu^\prime}, R, \mathbb{A})$ (see
Theorem 4.60 page 143 in \cite{jacod}).~Then, since $i$ is
arbitrary, by the same tools and a symmetric argument, the vector
integral with respect to $\boldsymbol{\mu^\prime}$ of a fixed
process in $\mathcal{L}^2(\boldsymbol{\mu^\prime}, R, \mathbb{A})$
is $(R,\mathbb{A})$-orthogonal to the vector integral with respect
to $\boldsymbol{\mu}$ of a fixed process in $\mathcal{L}^2(\boldsymbol{\mu}, R, \mathbb{A})$.\bigskip\\
The second statement follows by a slight generalization of Theorem 36, Chapter IV of \cite{prott}, which implies the result
when $\boldsymbol{\mu}, \boldsymbol{\mu^{\prime}}$ and $\boldsymbol{\mu^{\prime\prime}}$ are real martingales.\\
Let $\mathcal{I}$ be the space of processes
$$\bold{H}\bullet \boldsymbol{\mu} + \bold{H^{\prime}}\bullet \boldsymbol{\mu^\prime} + \bold{H^{\prime\prime}}\bullet \boldsymbol{\mu^{\prime\prime}}$$
with $\bold{H}\in\mathcal{L}^2(\boldsymbol{\mu}, R, \mathbb{A})$,
$\bold{H^\prime}\in \mathcal{L}^2(\boldsymbol{\mu^\prime}, R,
\mathbb{A})$, $\bold{H^{\prime\prime}}\in
\mathcal{L}^2(\boldsymbol{\mu^{\prime\prime}}, R,
\mathbb{A})$.~The stable space $\mathcal{Z}^2(\boldsymbol{\mu},
\boldsymbol{\mu^\prime}, \boldsymbol{\mu^{\prime\prime}})$
contains $\mathcal{Z}^2(\boldsymbol{\mu})$,
$\mathcal{Z}^2(\boldsymbol{\mu^\prime})$ and $\mathcal{Z}^2(
\boldsymbol{\mu^{\prime\prime}})$ and therefore it contains
$\mathcal{I}$ (see Proposition 4.5 page 114 and Theorem 4.35 page
130 in \cite{jacod}).~Moreover $\mathcal{I}$ turns out to be
stable and then it coincides with $\mathcal{Z}^2(\boldsymbol{\mu},
\boldsymbol{\mu^\prime},
\boldsymbol{\mu^{\prime\prime}})$\footnote{Note that $\mathcal{I}$
contains all the stochastic integrals with respect to any
components of $\mu$, $\mu^\prime$ and $\mu^{\prime \prime}$
}.\bigskip\\
First of all we show that $\mathcal{I}$ is closed. To this end we
consider the application which maps
$\mathcal{L}^2(\boldsymbol{\mu}, R,
\mathbb{A})\times\mathcal{L}^2(\boldsymbol{\mu^\prime}, R,
\mathbb{A})\times\mathcal{L}^2(\boldsymbol{\mu^{\prime\prime}}, R,
\mathbb{A})$ into $\mathcal{M}^2(R, \mathbb{A})$ in this way
$$\left(\bold{H}, \bold{H^{\prime}}, \bold{H^{\prime\prime}}\right)\rightarrow \bold{H}\bullet \boldsymbol{\mu} + \bold{H^{\prime}}\bullet \boldsymbol{\mu^\prime} + \bold{H^{\prime\prime}}\bullet \boldsymbol{\mu^{\prime\prime}}.$$
We prove that this application is an isometry.~We then conclude
the proof by observing that $\mathcal{I}$ is the image through the
above application of $\mathcal{L}^2(\boldsymbol{\mu}, R,
\mathbb{A})\times\mathcal{L}^2(\boldsymbol{\mu^\prime}, R,
\mathbb{A})\times\mathcal{L}^2(\boldsymbol{\mu^{\prime\prime}}, R,
\mathbb{A})$, which is an Hilbert space with scalar product
\begin{align*}&<\left(\bold{H}, \bold{H^{\prime}}, \bold{H^{\prime\prime}}\right), \left(\bold{K}, \bold{K^{\prime}}, \bold{K^{\prime\prime}}\right)>_{\mathcal{L}^2(\boldsymbol{\mu}, R, \mathbb{A})\times\mathcal{L}^2(\boldsymbol{\mu^\prime}, R, \mathbb{A})\times\mathcal{L}^2(\boldsymbol{\mu^{\prime\prime}}, R, \mathbb{A})}:=\cr &
E^R\left[\int_0^T\,\bold{H}^{tr}_t\,C^{\mu}_t\,\bold{K}_t\,dB^{\boldsymbol{\mu}}_t\right]+
E^R\left[\int_0^T\,(\bold{H}^{\prime})^{tr}_t\,C^{\boldsymbol{\mu}^\prime}_t\,\bold{K^{\prime}}_t\,dB^{\boldsymbol{\mu^\prime}}_t\right]+
E^R\left[\int_0^T\,(\bold{H^{\prime\prime}})^{tr}_t\,C^{\mu^{\prime\prime}}_t\,\bold{K^{\prime\prime}}_t\,dB^{\boldsymbol{\mu^{\prime\prime}}}_t\right].
\end{align*}
We refer to (\ref{def-B and C}) for the notations in the addends of the right-hand side, so that
for example
$$B^{\boldsymbol{\mu}}_t:=\sum_{i=1}^r <\mu^i>_t^{R,\mathbb{A}}\;\;\;\;\;\;c^{\boldsymbol{\mu}}_{ij}(t):=\frac{d<\mu^i,\mu^j>_t^{R,\mathbb{A}}}{dB^{\boldsymbol{\mu}}_t},\;\;\;i,i\in(1,\ldots, r).$$
As a consequence
$$\|\left(\bold{H}, \bold{H^{\prime}}, \bold{H^{\prime\prime}}\right)\|^2_{\mathcal{L}^2(\boldsymbol{\mu}, R,
\mathbb{A})\times\mathcal{L}^2(\boldsymbol{\mu^\prime}, R, \mathbb{A})\times\mathcal{L}^2(\boldsymbol{\mu^{\prime\prime}}, R, \mathbb{A})}:=
\|\bold{H}\|^2_{\mathcal{L}^2(\boldsymbol{\mu}, R,
\mathbb{A})}+\|\bold{H^\prime}\|^2_{\mathcal{L}^2(\boldsymbol{\mu^\prime},
R,
\mathbb{A})}+\|\bold{H^{\prime\prime}}\|^2_{\mathcal{L}^2(\boldsymbol{\mu^{\prime\prime}},
R, \mathbb{A})}$$ with
$$\|\bold{H}\|^2_{\mathcal{L}^2(\boldsymbol{\mu}, R, \mathbb{A})}=E^R\left[\int_0^T\,\bold{H}^{tr}_t\,C^{\mu}_t\,\bold{H}_t\,dB^{\boldsymbol{\mu}}_t\right]$$
that is
$$\|\bold{H}\|^2_{\mathcal{L}^2(\boldsymbol{\mu}, R, \mathbb{A})}=E^R\big[\left[\bold{H}\bullet \boldsymbol{\mu}\right ]_T\big]=
E^R\big[\left<\bold{H}\bullet \boldsymbol{\mu}\right >^{R, \mathbb{A}}_T\big]=E^R\big[\left(\bold{H}\bullet \boldsymbol{\mu}\right)^2_T\big].$$
The first equality derives from the general formula (\ref{def-norm2}), the second and the last equalities derive from a characterization and
from the definition of the $(R,\mathbb{A})$-sharp variation process of $\bold{H}\bullet \boldsymbol{\mu}$ respectively.\bigskip\\
Similarly
$$\|\bold{H^{\prime}}\|^2_{\mathcal{L}^2(\boldsymbol{\mu^\prime}, R, \mathbb{A})}=E^R\big[\left(\bold{H^{\prime}}\bullet \boldsymbol{\mu^\prime}\right)^2_T\big],
\;\;\;\;\;\;\|\bold{H^{\prime}}\|^2_{\mathcal{L}^2(\boldsymbol{\mu^{\prime\prime}}, R, \mathbb{A})}=E^R\big[\left(\bold{H^{\prime\prime}}\bullet \boldsymbol{\mu^{\prime\prime}}\right)^2_T\big].$$
Therefore it holds
$$\|\left(\bold{H}, \bold{H^{\prime}}, \bold{H^{\prime\prime}}\right)\|^2_{\mathcal{L}^2(\boldsymbol{\mu}, R, \mathbb{A})\times\mathcal{L}^2(\boldsymbol{\mu^\prime},
R, \mathbb{A})\times\mathcal{L}^2(\boldsymbol{\mu^{\prime\prime}}, R, \mathbb{A})}=E^R\big[\left(\bold{H}\bullet \boldsymbol{\mu}\right)^2_T+
\left(\bold{H^{\prime}}\bullet \boldsymbol{\mu^{\prime}}\right)^2_T+\left(\bold{H^{\prime\prime}}\bullet \boldsymbol{\mu^{\prime\prime}}\right)^2_T\big].$$
At the same time (see (\ref{def-norm3}))
\begin{align*}&\|\bold{H}\bullet \boldsymbol{\mu} + \bold{H^{\prime}}\bullet \boldsymbol{\mu^\prime} + \bold{H^{\prime\prime}}\bullet
\boldsymbol{\mu^{\prime\prime}}\|^2_{\mathcal{M}^2(R, \mathbb{A})}= E^R\left[\Big(\left(\bold{H}\bullet \boldsymbol{\mu}\right)_T +
 \left(\bold{H^\prime}\bullet \boldsymbol{\mu^\prime}\right)_T + \left(\bold{H^{\prime\prime}}\bullet \boldsymbol{\mu^{\prime\prime}}\right)_T\Big)^2\right].\end{align*}
By point i)
$$\bold{H}\bullet \boldsymbol{\mu} \cdot \bold{H^\prime}\bullet \boldsymbol{\mu^\prime},\;\;\;\;\bold{H}\bullet \boldsymbol{\mu} \cdot \bold{H^{\prime\prime}}\bullet
\boldsymbol{\mu^{\prime\prime}},\;\;\;\;\bold{H^\prime}\bullet \boldsymbol{\mu^\prime} \cdot \bold{H^{\prime\prime}}\bullet \boldsymbol{\mu^{\prime\prime}}$$
are centered $(R, \mathbb{A})$-martingales and in particular
$$E^R\left[(\bold{H}\bullet \boldsymbol{\mu})_T \cdot (\bold{H^\prime}\bullet \boldsymbol{\mu^\prime})_T\right]=
E^R\left[(\bold{H}\bullet \boldsymbol{\mu})_T \cdot (\bold{H^{\prime\prime}}\bullet \boldsymbol{\mu^{\prime\prime}})_T\right]=
E^R\left[(\bold{H^\prime}\bullet \boldsymbol{\mu^\prime})_T \cdot (\bold{H^{\prime\prime}}\bullet \boldsymbol{\mu^{\prime\prime}})_T\right]=0.$$
Then
\begin{align*}&\|\bold{H}\bullet \boldsymbol{\mu} + \bold{H^{\prime}}\bullet \boldsymbol{\mu^\prime} + \bold{H^{\prime\prime}}\bullet \boldsymbol{\mu^{\prime\prime}}\|^2_{\mathcal{M}^2(R, \mathbb{A})}=
E^R\big[\left(\bold{H}\bullet \boldsymbol{\mu}\right)^2_T+
\left(\bold{H^{\prime}}\bullet \boldsymbol{\mu^{\prime}}\right)^2_T+\left(\bold{H^{\prime\prime}}\bullet \boldsymbol{\mu^{\prime\prime}}\right)^2_T\big]\end{align*}
and therefore we get
$$\|\left(\bold{H}, \bold{H^{\prime}}, \bold{H^{\prime\prime}}\right)\|^2_{\mathcal{L}^2(\boldsymbol{\mu}, R, \mathbb{A})\times\mathcal{L}^2(\boldsymbol{\mu^\prime}, R,
 \mathbb{A})\times\mathcal{L}^2(\boldsymbol{\mu^{\prime\prime}}, R, \mathbb{A})}=\|\bold{H}\bullet \boldsymbol{\mu} + \bold{H^{\prime}}\bullet \boldsymbol{\mu^\prime} +
  \bold{H^{\prime\prime}}\bullet \boldsymbol{\mu^{\prime\prime}}\|^2_{\mathcal{M}^2(R, \mathbb{A})}.$$
Then $\mathcal{I}$ is closed.~For proving that $\mathcal{I}$ is
stable it is sufficient to recall that its elements are sum of
elements
of stable subspaces.\bigskip\\
The uniqueness of the triplet follows observing that if it would exists a different triplet $\boldsymbol{\eta}, \boldsymbol{\eta^\prime}, \boldsymbol{\eta^{\prime\prime}}$ such that
$$\boldsymbol{\Theta}\bullet (\boldsymbol{\mu},\boldsymbol{\mu^\prime}, \boldsymbol{\mu^{\prime\prime}}) = \boldsymbol{\eta}\bullet \boldsymbol{\mu} + \boldsymbol{\eta^{\prime}}\bullet \boldsymbol{\mu^\prime} + \boldsymbol{\eta^{\prime\prime}}\bullet \boldsymbol{\mu^{\prime\prime}}$$
then
$$\|\left(\boldsymbol{\xi}\bullet \boldsymbol{\mu} + \boldsymbol{\xi^{\prime}}\bullet \boldsymbol{\mu^\prime} + \bold{\xi^{\prime\prime}}\bullet \boldsymbol{\mu^{\prime\prime}}\right) - \left(\boldsymbol{\eta}\bullet \boldsymbol{\mu} + \boldsymbol{\eta^{\prime}}\bullet \boldsymbol{\mu^\prime} + \boldsymbol{\eta^{\prime\prime}}\bullet \boldsymbol{\mu^{\prime\prime}}\right)\|^2_{\mathcal{M}^2(R, \mathbb{A})}=0.$$
The linearity of the vector integral and point i) would imply
$$E^R\big[\left((\boldsymbol{\xi}-\boldsymbol{\eta})\bullet \boldsymbol{\mu}\right)^2_T\big]+
E^R\big[\left((\boldsymbol{\xi^{\prime}}-\boldsymbol{\eta^{\prime}})\bullet
\boldsymbol{\mu^{\prime}}\right)^2_T\big]+E^R\big[\left((\boldsymbol{\xi^{\prime\prime}}-\boldsymbol{\eta^{\prime\prime}})\bullet
\boldsymbol{\mu^{\prime\prime}}\right)^2_T\big]=0.$$
\end{proof}
\begin{remark}\label{rem_orth stable spaces}
Taking into account that the stable space generated by a martingale coincides with the set of vector integrals
 with respect to the martingale, we can summarize point i) and point ii) of the above lemma as follows
$$\mathcal{Z}^2(\mu^\prime)\subset \mathcal{Z}^2(\mu)^\perp$$
$$\mathcal{Z}^2(\mu, \mu^{\prime}, \mu^{\prime\prime})=\mathcal{Z}^2(\mu)\oplus \mathcal{Z}^2(\mu^{\prime})\oplus \mathcal{Z}^2(\mu^{\prime\prime}).$$
\end{remark}
\begin{remark}\label{rem_estensione_mult}
We remark that point ii) of Lemma \ref{lemma_orth stable spaces}
holds when considering a finite number $d$ of multidimensional
$(R,\mathbb{A})$-martingales $\mu^i$, $i=1,...,d$ with mutually
pairwise strongly orthogonal components.~More precisely the
following equality holds
$$\mathcal{Z}^2(\mu^{1},..., \mu^{d})=\oplus_{i=1}^d \mathcal{Z}^2(\mu^{i}).$$
\end{remark}


%
\begin{thm}\label{th-main}
Assume \textbf{A1)}, \textbf{A2)} and \textbf{A3)}.~Then
\begin{itemize}
\item [i1)] $\mathcal{F}_T$ and $\mathcal{H}_T$ are
$P$-independent; \item [i2)] $\mathbb{G}$ fulfills the standard hypotheses; \item [i3)]
every $W$ in $\mathcal{M}^2(P,\mathbb{G})$ can be uniquely
represented as
\begin{align}\label{equation-first-basis}
W_t=W_0+ (\boldsymbol{\gamma}^W \bullet \bold{M})_t+ (\boldsymbol{\kappa}^W\bullet \bold{N})_t+
(\boldsymbol{\phi}^W\bullet \bold{[M,N]^V})_t,\;\;\;P\textrm{-a.s.}~
\end{align}
with $\boldsymbol{\gamma}^W$ in $\mathcal{L}^2(\bold{M}, P, \mathbb{G})$, $\boldsymbol{\kappa}^W$
in $\mathcal{L}^2(\bold{N}, P, \mathbb{G})$ and $\boldsymbol{\phi}^W$ in
$\mathcal{L}^2(\bold{[M,N]^V}, P, \mathbb{G})$;
\item[i4)] there exists a probability measure $Q$ on $(\Omega,
\mathcal{G}_T)$ such that $(\bold{X},\bold{Y},\bold{[X,Y]^V})$
enjoys the $(Q, \mathbb{G})$-p.r.p..~More precisely every
 $Z$ in $\mathcal{M}^2(Q,\mathbb{G})$ can be uniquely represented as
\begin{align*}
Z_t=Z_0+(\boldsymbol{\eta}^Z\bullet \bold{X})_t + (\boldsymbol{\theta}^Z\bullet \bold{Y})_t +
(\boldsymbol{\zeta}^Z\bullet \bold{[X,Y]^V})_t\;\;\;Q\textrm{-a.s.},
\end{align*}
with $\boldsymbol{\eta}^Z$ in $\mathcal{L}^2(\bold{X}, Q, \mathbb{G})$, $\boldsymbol{\theta}^Z$ in
$\mathcal{L}^2(\bold{Y}, Q, \mathbb{G})$ and $\boldsymbol{\zeta}^Z$ in
$\mathcal{L}^2(\bold{[X,Y]^V}, Q, \mathbb{G})$.
\end{itemize}
\end{thm}
\begin{proof}
\begin{itemize}
\item [i1)] This statement is the extension to the multidimensional case of Lemma 4.2 in \cite{caltor15}.~Thanks to point i) of Lemma \ref{lemma_orth stable spaces} its
 proof is exactly the same.~For the sake of completeness we repeat it here.\bigskip\\
Proposition \ref{prop-pred-mart-rep} implies that if $A\in \mathcal{F}_T$ and $B\in \mathcal{H}_T$ then
 \begin{align}\label{rappA}
\mathbb{I}_A=P(A)+(\boldsymbol{\xi}^A\bullet \bold{M})_T,\;\;\;\;
\mathbb{I}_B=P(B)+(\boldsymbol{\xi}^B\bullet \bold{N})_T,\;\;\;\; P\textrm{-a.s.\;}
\end{align}
for $\boldsymbol{\xi}^A$ and $\boldsymbol{\xi}^B$  in
$\mathcal{L}^2(\bold{M},P,\mathbb{F})$ and
$\mathcal{L}^2(\bold{N},P,\mathbb{H})$ respectively.~These
equalities imply that $P(A\cap B)$ differs from $P(A)P(B)$ by the
expression
\begin{align*}
P(B)E^P\left[(\boldsymbol{\xi}^A\bullet \bold{M})_T\right]+P(A)E^P\left[(\boldsymbol{\xi}^B\bullet \bold{N})_T\right]+
E^P\left[(\boldsymbol{\xi}^A \bullet \bold{M})_T\,\cdot (\boldsymbol{\xi}^B\bullet \bold{N})_T\right].
\end{align*}
The above expression is null.~In fact the
$(P,\mathbb{G})$-martingale property of $\bold{M}$ and $\bold{N}$ and the
integrability of the integrands $\boldsymbol{\xi}^A$ and $\boldsymbol{\xi}^B$ imply that
the processes $\boldsymbol{\xi}^A\bullet \bold{M}$ and
$\boldsymbol{\xi}^B\bullet \bold{N}$ are real centered
 martingales.~Moreover, thanks to the assumption \textbf{A3)}, by suitably applying point i) of Lemma \ref{lemma_orth stable spaces},
 we get that the product $\boldsymbol{\xi}^A\bullet \bold{M}\cdot \boldsymbol{\xi}^B\bullet \bold{N}$ is a centered real
$(P,\mathbb{G})$-martingale.
\item [i2)] This is a direct consequence of the previous point and Lemma 2.2 in \cite{ame-be-schw03}.
\item [i3)] The proof of this point will
be done in three steps: \\
(a) the first goal is to prove the $(P,\mathbb{G})$-p.r.p.~for
$(\bold{M}, \bold{N}, \bold{[M,N]^V})$; \\
(b) as a second step the following key result is proved: fixed $i\in (1,\ldots, m)$ and $j\in (1,\ldots, n)$
the martingale $[M^i,N^j]$
is $(P,\mathbb{G})$-strongly orthogonal to $M^l$ and to
$N^h$ for arbitrary $l\in (1,\ldots, m)$ and $h\in (1,\ldots, n)$;\\
(c) finally point (a) and point (b) together with the second part
of Lemma \ref{lemma_orth stable spaces} allow to derive the
result.
\bigskip\\
(a) The $(P,\mathbb{G})$-p.r.p.~for $(\bold{M},\bold{N},\bold{[M,N]^V})$ is achieved
by proving that
$$\mathbb{P}((\bold{M},\bold{N},\bold{[M,N]^V}),\mathbb{G})=\{P\},$$ or, equivalently, that
for any $R\in\mathbb{P}((\bold{M},\bold{N},\bold{[M,N]^V}),\mathbb{G})$,  $P$ and $R$
coincide on the $\pi$-system $$\{A\cap B, \ A\in\mathcal{F}_{T}, \
B\in\mathcal{H}_{T}\},$$ which generates $\mathcal{G}_T$.~To this
end, note that the equalities in (\ref{rappA}) hold under $R$ so that
$R(A\cap B)$ differs from $P(A)P(B)$ by the expression
\begin{align}
P(B)E^R\left[(\boldsymbol{\xi}^A\bullet \bold{M})_T\right]+P(A)E^R\left[(\boldsymbol{\xi}^B\bullet \bold{N})_T\right]+
E^R\left[(\boldsymbol{\xi}^A\bullet \bold{M})_T\cdot(\boldsymbol{\xi}^B\bullet \bold{N})_T\right].
\end{align}
The above expression is null.~In fact \textbf{A1)} implies
$R|_{\mathcal{F}_T}=P|_{\mathcal{F}_T}$ and
$R|_{\mathcal{H}_T}=P|_{\mathcal{H}_T}$ and togheter with i1) this
in turn implies that $\boldsymbol{\xi}^A\bullet \bold{M}$ and
$\boldsymbol{\xi}^B\bullet \bold{N}$ are centered
$(R,\mathbb{G})$-martingales.~Moreover, by definition of $R$, for
all $i\in (1,\ldots, m)$ and $j\in (1,\ldots, n)$ the process
$[M^i,N^j]$ is a $(R,\mathbb{G})$-martingale so that by point i)
of Lemma \ref{lemma_orth stable spaces} the product
$\boldsymbol{\xi}^A\bullet \bold{M}\cdot \boldsymbol{\xi}^B\bullet
\bold{N}$
is a centered real $(R,\mathbb{G})$-martingale.\bigskip\\
%
(b) $[M^i,N^j]$ is $(P,\mathbb{G})$-strongly orthogonal to the
$(P,\mathbb{G})$-martingales $M^l$ and $N^h$, if and only if
$\left[M^l,[M^i,N^j]\right]$ and $\left[N^h,[M^i,N^j]\right]$ are uniformly
integrable $(P,\mathbb{G})$-martingales.\\Recall that
\begin{equation}\label{quadratic_covariation}
[M^i,N^j]_t=\langle M^{c,i},N^{c,j}\rangle^{P,\mathbb{G}}_t+\sum_{s\le t}\Delta M^i_s\Delta
N^j_s,\;\;\;\; P\textrm{-a.s.\;}
\end{equation}
where $M^{c,i}$ and $N^{c,j}$ are the $i$-component of the continuous martingale part of $\bold{M}$
and the $j$-component of the continuous martingale part of $\bold{N}$ respectively.~By point i1) $M^{c,i}$ and $N^{c,j}$
are independent $(P,\mathbb{G})$-martingales so that $\langle
M^{c,i},N^{c,j}\rangle^{P,\mathbb{G}}\equiv 0$, since by definition  $\langle
M^{c,i},N^{c,j}\rangle^{P,\mathbb{G}}$ is the unique $\mathbb{G}$-predictable process
with finite variation such that $M^{c,i}\,N^{c,j}-\langle M^{c,i},N^{c,j}\rangle^{P,\mathbb{G}}$
is $(P,\mathbb{G})$-local martingale equal to 0 at time 0 (see
Subsection 9.3.2. in \cite{Jean-yor-chesney}).~Therefore
\begin{equation}\label{eq-covar}
[M^{i},N^{j}]_t=\sum_{s\le t}\Delta M^i_s\Delta N^j_s.
\end{equation}
As a consequence
$$
\left[M^l,[M^i,N^j]\right]_t=\sum_{s\le t}\Delta M^l_s\Delta M^i_s\Delta N^j_s.
$$
Then for $u\le t$ one has
\begin{align*}
&E^P\left[\left[M^l,[M^i,N^j]\right]_t\mid \mathcal{G}_u\right]\\&=E^P\left[\sum_{s\le
u}\Delta M^l_s\Delta M_s^i\Delta N^j_s\mid\mathcal{G}_u\right]+E^P\left[\sum_{u<s\le t}\Delta M^l_s\Delta M_s^i\Delta N^j_s\mid\mathcal{G}_u\right]\\
&=\left[M^l,[M^i,N^j]\right]_u+\sum_{u<s\le t}E^P\left[\Delta M^l_s\Delta M_s^i\Delta N^j_s\mid\mathcal{G}_u\right]\\
&=\left[M^l,[M^i,N^j]\right]_u+\sum_{u<s\le t}E^P\left[\Delta M^l_s\Delta M_s^i\mid\mathcal{F}_u\right]E^P\left[\Delta
N^j_s\mid\mathcal{H}_u\right],
\end{align*}
where the last equality follows by point i1) and Lemma 4.3 in
\cite {caltor15}.~The  $(P,\mathbb{G})$-martingale property for
$[M^l,[M^i,N^j]]$ follows by observing that $E^P\left[\Delta
N^j_s|\mathcal{H}_u\right]=0$, for any $s>u$.~Finally
$\left[M^l,[M^i,N^j]\right]$ is uniformly integrable, since it is a $(P,\mathbb{G})$-regular martingale.\\
Analogously one gets that $[M^i,N^j]$ is $(P,\mathbb{G})$-strongly
orthogonal to $N^h$.\bigskip\\
(c) Point (a) implies that for every
$W\in\mathcal{M}^2(P,\mathbb{G})$ there exists a process
$\boldsymbol{\Theta}^W$ in $\mathcal{L}^2((\bold{M}, \bold{N},
\bold{[M,N]^V}), P, \mathbb{G})$ such that
$$W_t=W_0 + \left(\boldsymbol{\Theta}^W\bullet (\bold{M}, \bold{N}, \bold{[M,N]^V})\right)_t.$$
Taking into account point b), we then apply point ii) of Lemma
\ref{lemma_orth stable spaces} with $\mu=\bold{M},
\mu^\prime=\bold{N}, \mu^{\prime\prime}=\bold{[M,N]^V}$ and we
obtain immediately the thesis.
\item [i4)] Define
 $Q$ on $(\Omega,\mathcal{G}_T)$ by
$$\frac{dQ}{dP}:=L^{\bold{X}}\cdot L^{\bold{Y}}$$
where
$$L^{\bold{X}}:=\frac{dP^{\bold{X}}}{dP|_{\mathcal{F}_T}},\;\;\;\;\;\;L^Y:=\frac{dP^{\bold{Y}}}{dP|_{\mathcal{H}_T}}.$$
The definition is well-posed since by point i1) $L^{\bold{X}}\cdot L^{\bold{Y}}$ is
in $L^1(\Omega, P,\mathcal{G}_T)$.~$L^{\bold{X}}$ and $L^{\bold{Y}}$ are strictly
positive and therefore  $Q$ and $P|_{\mathcal{G}_T}$ are
equivalent measures.~Moreover for all $A$ in $\mathcal{F}_T$ and
$B$ in $\mathcal{H}_T$ it holds
$$
    Q(A\cap B)=E^P[\mathbb{I}_A\,L^{\bold{X}}]\,E^P[\mathbb{I}_B\,
    L^{\bold{Y}}],
$$
since $\mathcal{F}_T$ and $\mathcal{H}_T$ are independent under
$P$.~Using the equalities $E^P[L^{\bold{X}}]$ =1=
$E^P[L^{\bold{Y}}]$ one immediately gets the $Q$-independence of $\mathcal{F}_T$ and $\mathcal{H}_T$.\\
Finally $\bold{X}$ is a $(Q,\mathbb{F})$-martingale since
$Q|_{\mathbb{F}}=P^{\bold{X}}$ and it is also a $(Q,\mathbb{G})$-martingale
 by the $Q$-independence of $\mathbb{F}$ and $\mathbb{H}$.~Analogously it can be shown that  $\bold{Y}$ is a
$(Q,\mathbb{G})$-martingale.~Moreover the $Q$-independence of
$\mathcal{F}_T$ and $\mathcal{H}_T$ implies the analogous of point
(b) for $\bold{X}$ and $\bold{Y}$, that is the
$(Q,\mathbb{G})$-strong orthogonality of $X^i$ and $Y^j$, for all
$i\in (1,\ldots, m)$ and $j\in (1,\ldots, n)$.~The representation
i4) then follows by using the same procedure as in point i3).
\end{itemize}
\end{proof}
\section{Two applications}\label{sec:applications}

In this section we discuss two applications of Theorem \ref{th-main}.\bigskip\\
The first application is the extension of the representation
property i3) of Theorem \ref{th-main} to the case
$$\mathbb{G}:=\mathbb{F}^1\vee\ldots \vee\mathbb{F}^d$$
where, for all $i=1,\ldots ,d$, $\mathbb{F}^i\subset\mathcal{F}$ is
the reference filtration on $(\Omega, \mathcal{F}, P)$ of a real square integrable martingale
$M^i$ enjoying the $(P,\mathbb{F}^i)$-p.r.p.\bigskip\\
The second application proposes a  martingale representation
result closed to that given in the second part of Proposition 5.3
in \cite{ca-jean-za13}.~The statement dealt with the
representation under the historical measure  $P$ of every
square-integrable martingale of a market with default time $\tau$
when the available information was completed by the observation of
the default occurrence.~Here we work under very similar
hypotheses, but we also assume the \textit{immersion property} of
the filtration of the market $\mathbb{F}$ into the filtration
$\mathbb{G}=\bigcap_{s
> \cdot}\mathcal{F}_s\vee\sigma(\tau\wedge s)$, that is every
$(P,\mathbb{F})$-square-integrable martingale is a
$(P,\mathbb{G})$-square-integrable martingale too.

\subsection{The case of the union of a finite number of filtrations}\label{sec:componentwise}

First of all we prove a martingale representation result for a
reference filtration which is the union of just three
filtrations.~Then we extend it to the case of a reference
filtration which is the union of any finite number of filtrations.

\begin{thm}\label{thm-three filtrations}
Let  $\mathbb{F}^1, \mathbb{F}^2, \mathbb{F}^3$ be  three
filtrations on the space $(\Omega, \mathcal{F}, P)$.~For $i=1, 2,
3$, let $M^i$,
 be a real square integrable
$(P,\mathbb{F}^i)$-martingale.~Assume that
\begin{itemize}
\item [\textbf{B1)}] $\mathbb{P}(M^i,\mathbb{F}^i)=\{P_{|\mathcal{F}^i_T}\},\ \ i=1, 2, 3 \ $;
\item [\textbf{B2)}] for all pair (i,j) with $i, j\in\{1, 2, 3\}$, $M^i$ and $M^j$ are
$(P,\mathbb{G})$-strongly orthogonal martingales where
$$\mathbb{G}:=\mathbb{F}^1\vee\mathbb{F}^2\vee\mathbb{F}^3\ ;$$
\item [\textbf{B3)}]  $M^3$ is
$(P,\mathbb{G})$-strongly orthogonal to the $(P,\mathbb{G})$-martingale $[M^1, M^2]$.
\end{itemize}
Then
\begin{itemize}
\item [j1)] $\mathcal{F}^1_T, \mathcal{F}^2_T, \mathcal{F}^3_T$ are $P$-independent $\sigma$-algebras;
\item [j2)] $\mathbb{G}$ fulfills the standard hypotheses;
\item [j3)]
every $W$ in $\mathcal{M}^2(P,\mathbb{G})$ can be uniquely
represented $P$-a.s.~as
\begin{align*}
W_t=W_0+ \sum_{i=1}^3\int_0^t\Upsilon^{W, i}_s\,dM^i_s + \sum_{i,j
\in (1,2,3), i<j}\int_0^t\Phi^{W,i,j}_s\,d[M^i, M^j])_s+
\int_0^t\Psi^W_s\,d[[M^1,M^2],M^3]_s
\end{align*}
with $\Upsilon^{W,i}$ in $\mathcal{L}^2(M^i, P, \mathbb{G})$, $\Phi^{W,i,j}$
in $\mathcal{L}^2([M^i,M^j], P, \mathbb{G})$ and $\Psi^W$ in
$\mathcal{L}^2([[M^1,M^2],M^3], P, \mathbb{G})$.\\In particular the family
$$\left(M^1, M^2, M^3, [M^1,M^2], [M^1,M^3], [M^2,M^3], [[M^1,M^2],M^3]\right)$$ is a
$(P,\mathbb{G})$-basis of real strongly orthogonal martingales.
\end{itemize}
\end{thm}
\begin{proof}
\begin{itemize}
\item [j1)]
In order to show that $\mathcal{F}^1_T, \mathcal{F}^2_T, \mathcal{F}^3_T$ are $P$-independent we
observe that for any choice of $A^1\in \mathcal{F}^1_T$, $A^2\in \mathcal{F}^2_T$ and $A^3\in \mathcal{F}^3_T$ the value $P(A^1\cap A^2\cap A^3)$
 differs from $P(A^1)P(A^2)P(A^3)$ by the expectation under $P$ of the expression
\begin{align}\label{expression}
&\sum_i\,k_i\,\int_0^T\xi^i_s\,dM^i_s+\sum_{i,j\in (1,2,3),
i<j}\,k_{ij}\,\int_0^T\xi^{i}_s\,dM^i_s\cdot
\,\int_0^T\xi^{j}_s\,dM^j_s\cr &+
\,\int_0^T\xi^{1}_s\,dM^1_s\cdot\,
\,\int_0^T\xi^{2}_s\,dM^2_s\cdot \,\int_0^T\xi^{3}_s\,dM^3_s
\end{align}
where, for $i=1,2,3$, $\xi^i\in\mathcal{L}^2(M^i,P,\mathbb{F}^i)$ is the process, whose existence follows by assumption \textbf{B1)}, such that
     \begin{align}\label{rappAmult}
\mathbb{I}_{A^i}=P(A^i)+\,\int_0^T\xi^i_s\,dM^i_s,\;\;\;P\textrm{-a.s.\;}
\end{align}
and $k_i$, for $i=1,2,3$, and $k_{ij}$, for $i,j\in (1,2,3), i<j$, are suitable constants.\bigskip\\
But actually the expectation under $P$ of (\ref{expression}) is null.\\
In fact, first of all, for all $i=1,2,3$, the process $\int_0^{\cdot} \xi^{i}_s\,dM^i_s$ is an element of $\mathcal{Z}^2(M^i)$ and therefore it is a centered $(P,\mathbb{F}^i)$-martingale, so that
$$E^P\Big[\sum_i\,\int_0^T\xi^i_s\,dM^i_s\Big]=0.$$
Moreover assumption \textbf{B2)} joint with the first part of
Lemma \ref{lemma_orth stable spaces} provides  the product of
$\int_0^{\cdot} \xi^{i}_s\,dM^i_s \cdot \,\int_0^{\cdot}
\xi^{j}_s\,dM^j_s$ to be  a centered real
$(P,\mathbb{F}^i\vee\mathbb{F}^j)$-martingale for all $i,j\in
(1,2,3)$ with $i\neq j$, so that
$$E^P\Big[\sum_{i,j\in (1,2,3), i<j}\,\int_0^T\xi^{i}_s\,dM^i_s\,\,\cdot \int_0^T\xi^{j}_s\,dM^j_s\Big]=0.$$
Finally, by Theorem \ref{th-main} with $m=n=1$ and $X=M^1$,
$Y=M^2$, we derive that
$(P,\mathbb{F}^1\vee\mathbb{F}^2)$-martingale $\int_0^{\cdot}
\xi^{1}_s\,dM^1_s \cdot \,\int_0^{\cdot} \xi^{2}_s\,dM^2_s$
belongs to $\mathcal{Z}^2(M^1, M^2, [M^1, M^2])$.~From assumptions
\textbf{B2)} and \textbf{B3)} the martingale $M^3$ is
$(P,\mathbb{G})$-strongly orthogonal to $M^1, M^2$ and $[M^1,
M^2]$ so that any element of  $\mathcal{Z}^2(M^3)$ is
$(P,\mathbb{G})$-strongly orthogonal to all elements of
$\mathcal{Z}^2(M^1, M^2, [M^1, M^2])$ (see point a) of Theorem
4.7, page 116 in \cite{jacod}).~As a consequence $\int_0^{\cdot}
\xi^{1}_s\,dM^1_s \cdot\int_0^{\cdot} \xi^{2}_s\,dM^2_s \cdot
\,\int_0^{\cdot} \xi^{3}_s\,dM^3_s$ is a centered real
$(P,\mathbb{G})$-martingale so that
$$E^P\Big[\int_0^{T} \xi^{1}_s\,dM^1_s \cdot\int_0^{T} \xi^{2}_s\,dM^2_s \cdot \,\int_0^{T} \xi^{3}_s\,dM^3_s\Big]=0.$$
\item [j2)] This fact is a direct consequence of the previous point and Lemma 2.2 in \cite{ame-be-schw03}.
\item [j3)]
From  the assumptions \textbf{B2)} and \textbf{B3)} it follows immediately that the processes
$$ M^1, M^2, M^3, [M^1,M^2], [M^1,M^3], [M^2,M^3], [[M^1,M^2],M^3]$$
form a family of $(P,\mathbb{G})$-martingales.~It is easy to prove
that these martingales are pairwise $(P,\mathbb{G})$-strongly
orthogonal.~In fact  using j1) and  as in point (b) of the proof
of
 part i3) of Theorem \ref{th-main} we get that, fixed $i,j\in (1,2,3)$, $[M^i,M^j]$ coincides
 $P\textrm{-a.s.\;}$ with $\sum_{s\le t}\Delta M^i_s\Delta M^j_s$ and in particular $[M^i,M^j]$
  has no continuous martingale part.~As a consequence we derive that $\left[M^l,[M^i,M^j]\right]$, $l=1,2,3$,
   coincides with $\sum_{s\le t}\Delta M^l_s\Delta M^i_s\Delta M^j_s$ and, again using j1), we get that is a $(P,\mathbb{G})$-martingale,
that is $M^l$ is $(P,\mathbb{G})$-strongly orthogonal to $[M^i,M^j]$.~Analogously since
$\left[M^1,[M^2,M^3]\right]$ has no continuous martingale part, point j1) allows to prove that it is $(P,\mathbb{G})$-strongly
orthogonal to $M^k$ , for all $k=1,2,3$ and to
$[M^i,M^j]$, for all $i,j\in (1,2,3)$.\bigskip\\
From Theorem 36 in \cite{prott} it follows that
\begin{align*}&\mathcal{Z}^2\left( M^1, M^2, M^3, [M^1,M^2], [M^1,M^3], [M^2,M^3], [[M^1,M^2],M^3]\right)=\cr &
 \oplus_{i=1}^3\mathcal{Z}^2(M^i)\oplus_{i,j\in (1,2,3), i<j}\mathcal{Z}^2([M^i, M^j])\oplus\mathcal{Z}^2([[M^1, M^2],M^3])\end{align*}
 where the symbol of direct sum refers to the uniqueness of the representation of any martingale in $\mathcal{M}^2(P,\mathbb{G})$
 and to the orthogonality of the considered stable spaces, that is to the $(P,\mathbb{G})$-strong orthogonality of all pair of their elements.\bigskip\\
 We now show that
the vector martingale $$(M^1, M^2, M^3, [M^1,M^2], [M^1,M^3], [M^2,M^3], [[M^1,M^2],M^3])$$ enjoys the $(P,\mathbb{G})$-p.r.p..\bigskip\\
This shall imply that
$$
\mathcal{M}^2(P,\mathbb{G})=\oplus_{i\in (1,2,3)}\mathcal{Z}^2(M^i)\oplus_{i,j\in (1,2,3), i<j}\mathcal{Z}^2([M^i, M^j])\oplus\mathcal{Z}^2([[M^1, M^2],M^3])
$$
that is the thesis.\bigskip\\
Indeed we prove that
$$\mathbb{P}(M^1, M^2, M^3, [M^1,M^2], [M^1,M^3], [M^2,M^3], [[M^1,M^2],M^3]),\mathbb{G})=\{P\}.$$
In fact, if  $R\in\mathbb{P}(M^1, M^2, M^3, [M^1,M^2], [M^1,M^3],
[M^2,M^3], [[M^1,M^2],M^3]),\mathbb{G})$.\\Then  $P$ and $R$
coincide on the $\pi$-system $$\{A^1\cap A^2 \cap A^3, \
A^i\in\mathcal{F}^i_{T},\, i=1,2,3\},$$ which generates
$\mathcal{G}_T$.~To this end, note that the equalities in
(\ref{rappAmult}) hold under $R$-a.s.~so that $R(A^1\cap A^2 \cap
A^3)$ differs from $P(A^1)P(A^2)P(A^3)$ by the expectation under
$R$ of the expression (\ref{expression}).~The last turns out to be
null by the same arguments used in the proof of point j1) since
assumption \textbf{B1)} implies
$R|_{\mathcal{F}^i_T}=P|_{\mathcal{F}^i_T},\,i=1,2,3$.
\end{itemize}
\end{proof}
We now extend the theorem to the case of a general finite number $n>3$ of martingales.
\begin{thm}
Let $\mathbb{F}^1, \ldots, \mathbb{F}^d$ be filtrations on the
space  $(\Omega, \mathcal{F}, P)$. Let $M^i$ be a real square
integrable $(P,\mathbb{F}^i)$-martingale, for $i=1, \ldots,
d$.~Assume that
\begin{itemize}
\item [\textbf{C1)}] $\mathbb{P}(M^i,\mathbb{F}^i)=\{P_{|\mathcal{F}^i_T}\},\ i=1,\ldots, d$;
\item [\textbf{C2)}] for all $k\in (2,\ldots, d)$, for all set of index $(i_1,i_2,\ldots, i_k)$ with $i_1<i_2<\ldots i_k$,
$$[[[M^{i_1},M^{i_2}],M^{i_3}], \ldots M^{i_k}]$$ is a $(P,\mathbb{G})$-martingale where
$$\mathbb{G}:=\mathbb{F}^1\vee\mathbb{F}^2\ldots\vee\mathbb{F}^d.$$
\end{itemize}
Then the family obtained as the union of these sets of martingales\\
$M_i,\,i\in (1,2\ldots d)$\\
$[M^i,M^j],\,i,j\in (1,\ldots, d),\,i<j$,\\
$[[M^i,M^j],M^k],\,i,j,k\in (1,\ldots, d),\,i<j<k$,\\
$[[[M^i,M^j],M^k],M^l],\,i,j,k,l\in (1,\ldots, d),\,i<j<k<l$,\\
$ \ldots$,\\
$[[[[M^{1},M^{2}],M^3],M^4]\ldots,M^{d}]$\vspace{1em}\\
is a $(P,\mathbb{G})$-basis of real strongly orthogonal
martingales or equivalently
\begin{align*}
\mathcal{M}^2(P,\mathbb{G})=&\oplus_{i\in (1,2\ldots d)}\mathcal{Z}^2(M^i)\cr &\oplus_{i,j\in (1,2\ldots d), i<j}\mathcal{Z}^2([M^i, M^j])\cr
&\oplus_{i,j,k\in (1,2\ldots d), i<j<k}\mathcal{Z}^2([[M^i, M^j],M^k])\cr &\oplus_{i,j,k,l\in (1,2\ldots d), i<j<k<l}\mathcal{Z}^2([[[M^i, M^j],M^k],M^ l])\cr &\ldots\cr &\oplus\mathcal{Z}^2([[[M^1, M^2],M^3],\ldots,M^d]).
\end{align*}
\end{thm}
\begin{proof}
For $d\leq 3$ the result follows by  previous theorem.~For $d> 3$
 the key point is the
$P$-independence of the $\sigma$-algebras $\mathcal{F}^1_T,
\mathcal{F}^2_T,\ldots, \mathcal{F}^d_T$ that is, for any choice
of $A^1\in \mathcal{F}^1_T$, $A^2\in \mathcal{F}^2_T$ and $A^d\in
\mathcal{F}^d_T$,
 the factorization of $P(A^1\cap A^2\cap\ldots A^d)$.~In order to prove it we proceed by induction.~The
 basis
 of the induction is the $P$-independence of $\mathcal{F}^1_T$ and  $\mathcal{F}^2_T$, which derives immediately from j1)
  of previous theorem since the assumption \textbf{C1)} implies assumption \textbf{B1)} and assumption \textbf{C2)}
  implies assumption \textbf{B2)}.~Fixed $m$ in  $(4,\ldots, d-1)$, the inductive hypothesis is the $P$-independence
  of $\mathcal{F}^1_T, \mathcal{F}^2_T,\ldots, \mathcal{F}^m_T$.\bigskip\\
Following an  analogous proof  to that of point j3) of Theorem
\ref{thm-three filtrations} it can be easily showed
\begin{align}\label{representatio until m}
\mathcal{M}^2\left(P,\bigvee_{i=1}^{m}\mathbb{F}^i\right)=&\oplus_{i\in (1,2\ldots m)}\mathcal{Z}^2(M^i)
\cr &\oplus_{i,j\in (1,2\ldots m), i<j}\mathcal{Z}^2([M^i, M^j])\cr &\oplus_{i,j,k\in (1,2\ldots m), i<j<k}\mathcal{Z}^2([[M^i, M^j],M^k])
\cr &\oplus_{i,j,k,l\in (1,2\ldots m), i<j<k<l}\mathcal{Z}^2([[[M^i, M^j],M^k],M^ l])\cr &\ldots\cr &\oplus\mathcal{Z}^2([[[M^1, M^2],M^3],\ldots,M^m]).
\end{align}
%
%
Now we prove the $P$-independence of $\mathcal{F}^1_T,
\mathcal{F}^2_T,\ldots, \mathcal{F}^{m+1}_T$.~Fixed $A^1\in
\mathcal{F}^1_T$, $A^2\in \mathcal{F}^2_T$ and $A^{m+1}\in
\mathcal{F}^{m+1}_T$, obviously $P(A^1\cap A^2\cap\ldots \cap
A^{m+1})$ differs from $P(A^1)\cdot P(A^2)\cdot\ldots \cdot
P(A^{m+1})$ by the $P$-expectation of an expression containing, up
to multiplicative constants, terms of the form
\begin{align}\label{terms to m}
\int_0^T\xi^{i_1}_s\,dM^{i_1}_s\,\cdot\,\int_0^T\xi^{i_2}_s\,dM^{i_2}_s\,\cdot\,\int_0^T\xi^{i_p}_s\,dM^{i_p}_s
\end{align}
with $p\leq m$ and $ i_1<i_2<\ldots<i_p\in (1,2,\ldots ,m),$, where $\xi^{i_r}\in \mathcal{L}^2(M^{i_r},P,\mathbb{F}^{i_r}),\;  1\leq r\leq p$, and  one term of the form
\begin{align}\label{last term}
\int_0^T\xi^{1}_s\,dM^{1}_s\,\cdot\,\int_0^T\xi^{2}_s\,dM^{2}_s\,\cdot\,\int_0^T\xi^{m}_s\,dM^{m}_s\,\cdot\int_0^T\xi^{m+1}_s\,dM^{m+1}_s
\end{align}
where $\xi^{r}\in \mathcal{L}^2(M^{r},P,\mathbb{F}^{r}),\;  1\leq r\leq m+1$.\bigskip\\
By the the inductive hypothesis the terms of the form (\ref{terms to m}) are the final values of centered $(P,\mathbb{G})$-martingales so that they have null $P$-expectation.\bigskip\\
Moreover by (\ref{representatio until m}) the product of the first
$m$ integrals in (\ref{last term}), wich belongs to
$\mathcal{M}^2\left(P,\bigvee_{i=1}^{m}\mathbb{F}^i\right)$,
is a sum of integrals with respect to martingales of the family\\
$M_i,\,i\in (1,2\ldots m)$,\\
$[M^i,M^j],\,i,j\in (1,\ldots, m),\,i<j$,\\
$[[M^i,M^j],M^k],\,i,j,k\in (1,\ldots, m),\,i<j<k$,\\
$[[[M^i,M^j],M^k],M^l],\,i,j,k,l\in (1,\ldots, m),\,i<j<k<l$,\\
$ \ldots$,\\
$[[[[M^{1},M^{2}],M^3],M^4]\ldots,M^{m}]$.\vspace{1em}\\
Then, since by \textbf{C2)} the stable space generated by the above family is orthogonal to the stable space generated by $M^{m+1}$, (\ref{last term})
is the final value of a centered $(P,\mathbb{G})$-martingale so that its $P$-expectation is equal to zero.
\end{proof}
\begin{remark}\label{rem-duffie-davis}
We get that $2^d-1$ is the biggest possible value of the
\textit{multiplicity} of $\mathbb{G}$ in the sense of Davis
Varaiya  (see \cite{davis1}), that is $2^d$
 is the maximum \textit{spanning number} of the economy given by the assets $M^1,\ldots, M^d$ (see \cite{duffie86}).
\end{remark}
\subsection{An example of credit risk modeling}\label{sec:default}

The second part of
Proposition 5.3 in \cite{ca-jean-za13}
 dealt with the martingale representation under the market measure $P$
 when the reference filtration $\mathbb{F}$ was progressively
 enlarged by the occurrence of a default time $\tau$.~More precisely a $(P,\mathbb{G})$-basis of real strongly orthogonal martingales when
 $\mathcal{G}_t=\cap_{s>t}\mathcal{F}_s\vee\sigma(\tau\wedge
s)$ was derived.~The main hypotheses were the existence of a real
$(P,\mathbb{F})$-martingale which was a basis (not necessarily
continuous) for the $(P,\mathbb{F})$-local martingales and the
equivalence, for all $t$, between the $\mathcal{F}_t$-conditional
law of $\tau$ and a fixed deterministic probability measure on
$\mathbb{R}^+$ without atoms, $\nu$.~The key tool was the
decoupling martingale preserving measure introduced by Grorud and
Pontier in \cite{gro_po99} and by Amendinger in
\cite{ame}.\bigskip\\
In a previous paper (see \cite{jean_lecam09}) Jeanblanc and Le Cam
assumed the existence of a (possibly multidimensional) continuous
semi-martingale $S$ with the $(P^*,\mathbb{F})$-p.r.p., and
moreover they assumed \textit{Jacod hypothesis} (see
\cite{jacod85}), that is just the absolute continuity with respect
to $\nu$ of the $\mathcal{F}_t$-conditional law of $\tau$ for all
$t$.~For the sake of completeness we recall that a time $\tau$
under Jacod hypothesis is called an \textit{initial time},
independently of the nature of $\nu$.~In \cite{jean_lecam09} $\nu$
was equal to Lebesgue measure and the authors identified a basis
for the $(P^*,\mathbb{G})$-martingales (see Theorem 1 in
\cite{jean_lecam09} and the analogous Proposition 4.2 in
\cite{ca-jean-za13}).\bigskip\\
It is to stress that both papers
worked under the so-called \textit{density hypothesis}, that is
$\tau$ initial time and $\nu$ with no atoms (see
\cite{elka-jean-jiao09}).~As it is well-known this assumption
forces $\tau$ to have no atoms and to \textit{avoids
$\mathbb{F}$-stopping times}
 (for any finite $\mathbb{F}$-stopping time $T$ it holds $\mathbb{P}(\tau=T)=0$, see Proposition 1 in \cite{jean_lecam09}).\bigskip\\Both papers gave a kind of generalization of the
martingale representation result obtained by Kusuoka when
$\mathbb{F}$ is equal to the natural filtration of a Brownian
motion $B$ (see \cite{kusuoka99}).~To the density hypothesis
actually Kusuoka added the immersion property, which was the
necessary condition in order to get $B$ as an element of a
$(P,\mathbb{G})$-basis.~The second element of the
$(P,\mathbb{G})$-basis was the \textit{$\mathbb{F}$-conditional
compensated default process}.~Without the immersion hypothesis the
Brownian motion should have be substituted by the martingale part
of its $(P,\mathbb{G})$-semi-martingale
decomposition, analogously to what happened  in  \cite{ca-jean-za13} and \cite{jean_lecam09}.\\

We propose a similar result, when the asset is a general
multidimensional semi-martingale enjoying the p.r.p. and $\tau$
satisfies the same hypotheses as in Proposition 5.3 in
\cite{ca-jean-za13}.~Moreover we assume immersion property of
$\mathbb{F}$ in $\mathbb{G}$ under $P$.
\begin{proposition}
Given a filtered probability space $(\Omega,\mathcal{F},\mathbb{F},P)$, let $\tau$ be a continuous random time such that
\begin{equation}\label{equivalence}
P(\tau\in\cdot\mid\mathcal{F}_t)\sim
P(\tau\in\cdot),\,\textrm{for\, every}\;t\in
[0,T],\;P\textrm{-a.s}.
\end{equation}
and $\bold{X}$ be an $m$-dimensional
$(P,\mathbb{F})$-semi-martingale like in (\ref{semi-martingale})
satisfying hypotheses \textbf{A1)} and \textbf{A2)}.~Consider the
progressively enlarged  filtration $\mathbb{G}$ defined by
$$\mathcal{G}_t:=\cap_{s>t}\mathcal{F}_s\vee\sigma(\tau\wedge
s).$$

 Let $\lambda$ be the real process
\begin{equation}\label{intensity}\lambda_t=\frac{p_t(t)}{P(\tau>t\mid\mathcal{F}_t)}, \ \ t\in [0,T],\end{equation}

with the process $\left(p_t(u)\right)_{u}$ defined by
\begin{equation}\label{conditional
density}\int_t^{T}p_t(u)\,du=P(\tau>t\mid\mathcal{F}_t).\end{equation}

Assume  $\mathbb{F}$ to be $P$-immersed in $\mathbb{G}$.~Then the
pair
$$\left(\bold{M},\mathbb{I}_{\{\tau\leq \cdot\}}-\int_0^{\tau\wedge
\cdot}\lambda_u\,du\right)$$ is a $(P,\mathbb{G})$-basis of
multidimensional martingales.
\end{proposition}
\begin{proof}  Let $F$ be the continuous distribution function of $\tau$ and let $\mathbb{H}$ be the natural filtration of the
process  $\left(\mathbb{I}_{\{\tau\leq t\}}\right)_{t}$.~Let $N$ be defined by
$$N_t=\mathbb{I}_{\{\tau\leq t\}}-\int_0^{\tau\wedge t}\frac{dF_u}{1-F_u}.$$$N$ is a real $(P,\mathbb{H})$-martingale
  enjoying
 the $(P|_{\mathcal{H}_T},\mathbb{H})$-p.r.p.~(see  Proposition 7.2.2.1 and Proposition 7.2.5.1 in \cite{Jean-yor-chesney}).~\bigskip\\Moreover $\bold{M}$
  enjoys the $(P|_{\mathcal{F}_T},\mathbb{F})$-p.r.p.~(see Proposition \ref{prop-pred-mart-rep}).\bigskip\\Now, the equivalence (\ref{equivalence}) implies the existence of
   a probability measure $P^*$ on $(\Omega, \mathcal{G}_T)$ under which $\mathbb{F}$ and $\sigma(\tau)$ are
 independent and such that $P^*|_{\mathcal{F}_T}=P|_{\mathcal{F}_T}$ and  $P^*|_{\mathcal{H}_T}=P|_{\mathcal{H}_T}$ (see Proposition 3.1 in \cite{ame}).~Then $P^*$
 decouples $\mathbb{F}$ and $\mathbb{H}$, since $\mathcal{H}_T=\sigma(\tau)$, so that, for any $i=1,\ldots, m$, $M^i$ and $N$ are $(P^*,\mathbb{G})$-strongly orthogonal
martingales.~Therefore Theorem \ref{th-main} applies.~Indeed
 $N$ jumps at an $\mathbb{H}$-totally inaccessible time since the law of $\tau$ has no atoms (see Remark 7.2.1.2 in \cite{Jean-yor-chesney} and IV 107 in \cite{del-me-a})\footnote{we observe that the jump time of
 $N$ is also a $\mathbb{G}$-totally inaccessible time thanks to Lemma 3.5 in \cite{jean-coku-nike12}, however this is not important here} and under $P^*$ the filtrations $\mathbb{F}$ and $\mathbb{H}$ are independent so that $[M^i,N]\equiv 0$.~We conclude that $(\bold{M},N)$ is a $(P^*,\mathbb{G})$-multidimensional basis.\bigskip\\
Let us introduce
$$\widetilde{L}^*=\frac{dP|_{\mathcal{G}_T}}{dP^*}.$$
Now we define $\tilde{M}^i, i=1,\ldots, m,$ and $\tilde{N}$  by
$$
\tilde{M^i_t}:=M^i_t-\int_0^t\frac{1}{\widetilde{L}^*_{s^-}}\;d\langle
\widetilde{L}^*, M^i \rangle^{P^*,\mathbb{G}}_{s}, \ \ t\in [0,T],
$$
$$
\tilde{N_t}:=N_t-\int_0^t\frac{1}{\widetilde{L}^*_{s^-}}\;d\langle
\widetilde{L}^*, N \rangle^{P^*,\mathbb{G}}_{s}, \ \ t\in [0,T].
$$
We observe that the processes $\langle \widetilde{L}^*, M^i\rangle^{P^*,\mathbb{G}}_\cdot$ and
 $\langle \widetilde{L}^*, N\rangle^{P^*,\mathbb{G}}_\cdot$ exist (see e.g.~VII 39 in \cite{del-me-b}).~
By Lemma 2.4 in \cite{jean-song15} the pair $(\bold{\tilde{M}},\tilde{N})$ enjoys the
$(P,\mathbb{G})$-p.r.p.\bigskip\\
In order to identify the pair $(\bold{\tilde{M}},\tilde{N})$, we recall that, since $\mathbb{F}$ is $P$-immersed in $\mathbb{G}$,
the process $$\mathbb{I}_{\{\tau\leq \cdot\}}-\int_0^{\tau\wedge \cdot}\lambda_u\,du$$ with $\lambda_\cdot$ defined by (\ref{intensity})
and (\ref{conditional density}), is a $(P,\mathbb{G})$-martingale (see pag 429 of \cite{Jean-yor-chesney} or Proposition 6.3 in \cite{jean-rut-99}).\bigskip\\
Now the pair $(\bold{\tilde{M}},\tilde{N})$ coincides with
$\left(\bold{M},\mathbb{I}_{\{\tau\leq \cdot\}}-\int_0^{\tau\wedge
\cdot}\lambda_u\,du\right)$.\bigskip\\In fact, for $i=1,\ldots, m$
$$
\tilde{M}^i_t-M^i_t=\int_0^t\frac{1}{\widetilde{L}^*_{s^-}}\;d\langle
\widetilde{L}^*, M^i \rangle^{P^*,\mathbb{G}}_{s},\;\;\;t\in [0,T]
$$
and
\begin{align*}
&\tilde{N}_t- \left(\mathbb{I}_{\{\tau\leq t\}}-\int_0^{\tau\wedge
t}\lambda_u\,du\right)=\\& \int_0^{\tau\wedge t}\frac{dF_u}{1-F_u}
-\int_0^{\tau\wedge
t}\lambda_u\,du+\int_0^t\frac{1}{\widetilde{L}^*_{s^-}}\;d\langle
\widetilde{L}^*, N \rangle^{P^*,\mathbb{G}}_{s}, \ \ \ t\in [0,T]
\end{align*}
are predictable $(P,\mathbb{G})$-local martingales, so they have to
be continuous (see Theorem 43 Chapter IV in \cite
{kall97}).~Moreover they have finite variation, so that they are necessarily
null.\bigskip\\
In remains to prove that, for all $i=1,\ldots, m$, $M^i$ and
$\mathbb{I}_{\{\tau\leq \cdot\}}-\int_0^{\tau\wedge
\cdot}\lambda_u\,du$ are $(P,\mathbb{G})$-strongly orthogonal
 martingales.~We consider their quadratic-covariation process which coincides $P$-a.s.~and $P^*$-a.s.~with
$$\left[M^i,\mathbb{I}_{\{\tau\leq \cdot\}}\right]_t-\left[M^i,\int_0^{\tau\wedge \cdot}\lambda_u\,du\right]_t, \;\;\;t\in [0,T].$$
The first addend is null since $\tau$ avoids $\mathbb{F}$-stopping
times.~The second addend is null since $\tau$ totally inaccessible
implies that $\int_0^{\tau\wedge \cdot}\lambda_u\,du$ is
continuous and of finite variation.
\end{proof}
\section{Perspectives}\label{sec:perspectives}
The object of ongoing research is to deeper
  investigate the problem of martingale representation under the historical measure in markets driven by
processes sharing  accessible jumps times  with positive
probability.~In particular the authors conjecture to extend the
Kusuoka like representation result to the case of a default time
$\tau$ which doesn't satisfies the density hypothesis.

\bibliography{biblio}

\end{document}